\documentclass[reqno,12 pt]{amsart}

\usepackage{amsfonts,amscd,mathrsfs}
\usepackage{amsthm}
\usepackage{amssymb}
\usepackage{latexsym}
\usepackage{url}
\usepackage{cite}
\usepackage{amsmath}
\usepackage{verbatim}
\usepackage{graphicx}
\usepackage{epsf}
\usepackage{enumerate}
\usepackage{amsbsy}
\usepackage{color}
\usepackage{geometry}
\usepackage{letltxmacro}
\usepackage{hyperref}
\definecolor{dark-blue}{rgb}{0,0,0.6}
\definecolor{Purple}{rgb}{0.2,0,0.25}
\hypersetup{breaklinks=true,
pagecolor=white,
colorlinks=true,
citecolor=dark-blue,
filecolor=black,%
linkcolor=Purple,%
urlcolor=black
}
\newtheorem{thm}{Theorem}[section]
\newtheorem{cor}[thm]{Corollary}
\newtheorem{lem}[thm]{Lemma}

\newtheorem{defin}[thm]{Definition}

\theoremstyle{definition}
\newtheorem{expl}[thm]{Example}
\newtheorem{remark}[thm]{Remark}

\newcommand{\dom}{\textnormal{dom}}

\newcommand{\R}{\mathbb{R}}

\newcommand{\N}{\mathbb{N}}

\newcommand{\argmin}{\textnormal{argmin}}
\newcommand{\sign}{\textnormal{sign}}

\newcommand{\bref}[1]{\textbf{\ref{#1}}} 
\newcommand{\beqref}[1]{\textbf{(\ref{#1})}} 

\begin{document}
\date{June 29, 2016}

\title[A new convergence analysis]{ A new convergence analysis and perturbation resilience of some  accelerated proximal forward-backward algorithms with errors}
\author{Daniel Reem and Alvaro De Pierro}
\address{Instituto de Ci\^encias Matem\'aticas e de Computa\c{c}\~ao (ICMC), University of S\~ao Paulo,  
 S\~ao Carlos, SP, Brazil and Department of Mathematics, The Technion - Israel Institute of Tehcnology, Haifa 3200003, Israel}
\email{dream@tx.technion.ac.il}
\address{Instituto de Ci\^encias Matem\'aticas e de Computa\c{c}\~ao (ICMC), University of S\~ao Paulo, S\~ao Carlos, Avenida Trabalhador S\~ao-carlense, 400 - Centro, 
CEP: 13566-590, S\~ao Carlos, SP, Brazil. }
\email{depierro.alvaro@gmail.com}

\thanks{This work was supported by FAPESP 2013/19504-9. The second author was supported also by CNPq grant  306030/2014-4.}

\maketitle

\begin{abstract}
Many problems in science and engineering involve, as part of their  solution process, the consideration of a separable function which is the sum of two convex functions, one of them possibly non-smooth. Recently a few works have discussed inexact versions of several accelerated proximal methods aiming at solving this minimization problem. This paper shows that inexact versions of a method of Beck and Teboulle (FISTA) preserve, in a Hilbert space setting, the same (non-asymptotic) rate of convergence under some assumptions on the decay rate of the error terms.  The notion of inexactness discussed here seems to be rather simple, but, interestingly, when comparing to related works, closely related decay rates of the errors terms yield closely related  convergence rates. The derivation sheds some light on the somewhat mysterious origin of some parameters which appear in various accelerated methods. A consequence of the analysis is that the accelerated  method is perturbation resilient, making it suitable, in principle, for the superiorization methodology. By taking this into account, we re-examine the  superiorization methodology and significantly extend its scope. 
\end{abstract}

\section{Introduction}\label{sec:Intro}
\subsection{Background: }\label{subsec:Background} Many problems in science and engineering involve, as part of their  solution process, the consideration of the following minimization problem:
\begin{equation}\label{eq:Minimize}
\inf\{F(x): x\in H\}.
\end{equation}
Here $F$ is a separable function of the form $F=f+g$, both $f$ and $g$ are convex functions defined on a real Hilbert space $H$ (with an inner product $\langle\cdot,\cdot\rangle$ and an  induced norm $\|\cdot\|$), the function $g$  is lower  semicontinuous and possibly non-smooth, and $f$ is continuously differentiable and its derivative $f'$ is  Lipschitz continuous with a Lipschitz constant $L(f')\geq 0$. A typical scenario of \beqref{eq:Minimize} appears in linear inverse problems \cite{EnglHankeNeubaue1996rbook,Hansen1998book}. There $H=\R^n$,  $b\in \R^m$, $f(x)=\|Ax-b\|^2$ for some $m\times n$ matrix $A$,  $g(x)=\lambda\|Lx\|^2$, and $L$ is an $m\times n$ matrix (often $L$ is the identity operator, or a diagonal one, or a discrete  approximation of a differential operator). The dimensions $m$ and $n$ are large, e.g., on  the order of magnitude of $10^3$, and $\lambda$ is a fixed positive constant (the regularization parameter). The goal is to estimate the solution $x\in \R^n$ to the linear equation 
\begin{equation}\label{eq:LinearInverse}
Ax=b+u,
\end{equation}
where  $u\in \R^m$ is an unknown noise vector. The solution $x$ frequently represents an image or a signal and the consideration of \beqref{eq:Minimize} instead of \beqref{eq:LinearInverse} is motivated from the fact that \beqref{eq:LinearInverse} is often ill-conditioned. 

The  $\ell_1-\ell_2$  minimization problem (or closely related variations of it) is a variation of  the previous problem which has become popular in machine learning, compress sensing, and signal  processing  \cite{BachJenattonMairalObozinski2012jour,CandesWakinBoyd2008,ChenDonohoSaunders2001jour,Tibshirani1996jour}. Here one frequently takes $g(x)=\lambda\|Lx\|_1$ or $g(x)=\sum_{\nu\in \Pi}\lambda_{\nu}\|x_{\nu}\|_{\infty}$  where $\|x\|_1=\sum_{i=1}^n|x_i|$, $\Pi$ is a vector of  positive integers $\nu$ whose sum is $n$, $\|x_{\nu}\|_{\infty}=\max\{|x_j|: j\in \{1,\ldots,\nu\}\}$, and $\lambda_{\nu}>0$ for all components $\nu$ of $\Pi$. The non-smooth terms are used for increasing sparsity. As a final example we mention the nuclear norm approximation minimization problem which has several versions (and it includes, as a special case, the minimum rank matrix completion problem). In one version $x\in \R^{m\times n}$, $f(x)=\|Ax-b\|^2$, $A:\R^{m\times n}\to \R^{\ell}$ is linear, $b\in \R^{\ell}$, $g(x)=\lambda \|x\|_{\textnormal{nuc}}$, and $\|x\|_{\textnormal{nuc}}$ is the nuclear norm of $x$, i.e., the sum of singular values of $x$ where here $x$ is viewed as a matrix \cite{CaiCandeShen2010jour,MaGoldfarbChen2011jour} (the nuclear norm is  aimed at providing a convex approximation of the matrix rank function). In a second version  $x\in \R^n$, $f:\R^n\to\R$ is a quadratic function, $g(x)=\|Ax-B\|_{\textnormal{nuc}}$, $A:\R^n\to \R^{p\times q}$ is a linear mapping, and $B\in \R^{p\times q}$ is a given matrix \cite{LiuVandenberghe2009jour}. As discussed in the previous references and in some of the references therein, this problem has applications in control and system theory, compressed sensing, computer vision, data recovering, and more. 

Proximal (gradient) methods are among the methods used for solving \beqref{eq:Minimize}. Roughly speaking, they have the form 
\begin{equation}\label{eq:x_kQ_k}
x_k=\argmin_{x\in H}Q_k(x,y_k)
\end{equation}
where $Q_k:H^2\to\R$ is a sum of a two-variable quadratic function and of $g$ and it depends on the iteration $k$, on $f$, and possibly on some other parameters, and $y_k$ depends linearly on previous iterations. When $y_k=x_{k-1}$, a convergence of the iterative sequence $(x_k)_{k=1}^{\infty}$ to a solution $x^*$ of \beqref{eq:Minimize} (assuming such a solution exists) can be established, but unless some strong conditions are  imposed on $F$ and/or other components involved in the problem (e.g., properties of the solution set), both the asymptotic convergence $(x_k\xrightarrow[k\to \infty]{}x^*)$ and the non-asymptotic one ($F(x_k)\xrightarrow[k\to \infty]{}F(x^*))$ can be slow, e.g., $F(x_k)-F(x^*)=O(1/k)$. See, for instance, the discussions in  \cite{BeckTeboulle2009jour} about the ISTA method (Iterative Shrinkage  Tresholding Algorithm), and in \cite[Chapters 4-5]{Byrne2014book}, \cite{CombettesWajs2005jour} about related generalizations and variations. 

The above disadvantage is one of the reasons why accelerated proximal gradient methods are of interest, methods in which a non-asymptotic rate of convergence of the form $F(x_k)-F(x^*)=O(1/k^2)$ can be achieved.  The first significant achievements in this area seem to be the works of  Nemirovski and Yudin  \cite[Chapter 7]{NemirovskyYudin1983book} (1979), and Nemirovski  \cite{Nemirovskiy1982jour} (1982) (with ideas which go back to their 1977 paper  \cite{YudinNemirovskiy1977jour}), for the case of certain smooth functions (i.e., $g\equiv 0$) in a certain class of smooth real reflexive Banach spaces. However, their methods were rather complicated. A breakthrough occurred some time later (1983) by Nesterov  \cite{Nesterov1983jour},  who presented a simple and very practical accelerated method  for the case $g\equiv 0$ and $F$ defined on a Euclidean space. A few years ago there have been additional significant achievements  when the case of $F=f+g$ with a non-smooth $g$ has been discussed in a Euclidean space setting  by Nesterov \cite[Section 4]{Nesterov2007prep} in 2007 and Beck and  Teboulle \cite{BeckTeboulle2009jour} in 2009. Both papers improved independently (and  using different approaches) Nesterov's method \cite{Nesterov1983jour}  using clever modifications. Beck and Teboulle called their method FISTA (Fast Iterative  Shrinkable Tresholding Algorithm).  In the accelerated methods $y_k$ is not $x_{k-1}$ but a linear combinations of several previous iterations. For instance, in FISTA  $y_{k+1}=x_k+\beta_k(x_k-x_{k-1})$ and in Nesterov's method 
$y_{k+1}=\beta_kz_k+(1-\beta_k)x_k$, where $z_k$  is a minimizer of a one-variable 
quadratic function and $\beta_k$ is a positive parameter. For related accelerated 
methods, see e.g.,  \cite{AuslenderTeboulle2006jour,BeckTeboulle2009-2-jour,BelloCruz-Nghia2015prep,GoldfarbMaScheinberg2012jour,GonzagaKaras2013jour,LanLuMonteiro2011jour,NarkissZibulevsky2005prep,Nesterov2004book,Nesterov2005jour,Tseng2008prep,Tseng2010jour,ZibulevskyElad2010jour}. 

A natural question regarding these accelerated methods is whether they are perturbation resilient. In other words, are they stable, i.e., do they still exhibit an accelerated rate of convergence despite  perturbations which may appear in the iterative steps due to noise, computational errors, etc. 
The relevance of this question becomes even more evident when taking into account the fact that  the iterative step in these method involves a proximity operator (see \beqref{eq:x_kQ_k} and   \beqref{eq:p_L})  whose computation is likely to be inexact, since it is itself a solution to a minimization problem. 
Because of that, there has been a rather wide related discussion on inexact proximal forward-backward methods as the following partial list of references shows:  \cite{AlberBurachikIusem1997,BurachikSvaiter1999jour,BurkeQian1999jour,Combettes2004jour,CombettesWajs2005jour,Cominetti1997jour,Eckstein1998jour,Guler1992jour,HeYuan2012jour,HumesSilva2005jour,
IusemOtero2001jour,IusemPennanenSvaiter2003,KangKangJung2015jour,KaplanTichatschke2004jour,ParenteLotitoSolodov2008jour,Rockafellar1976jour,SalzoVilla2012jour,SantosSilva2014jour,SolodovSvaiter1999-2jour,SolodovSvaiter1999-1jour,SolodovSvaiter2000jour,XiaHuang2011jour,Zaslavski2010jour}. In these papers various  notions of inexactness and various settings are discussed (however,  in many cases the methods are non-accelerated, the functions are non-separable, and no convergence estimates are given). 

Another motivation to discuss  inexactness in relation to proximal methods is the recent optimization scheme called ``superiorization''  \cite{Censor2015surv,Davidi2010PhD,Herman2014surv}. In this scheme ones uses carefully selected  perturbations in an  active way in order to obtain solutions which have some good properties,  properties which are measured with respect to some auxiliary cost function (or energy/merit  function). For instance, if one wants to minimize a given function under some constraints, then  instead of  solving this problem which might be too demanding, one may try to find a point which satisfies the constraints but is not necessarily a minimizer. Instead, this point will have a low cost function value and hence it will be superior to other points which satisfy the constraints. See Section \bref{sec:Superiorization} for a more comprehensive discussion and many more related references. 

To the best of  our knowledge, the issue of inexactness related to accelerated proximal forward-backward methods with a separable function $F=f+g$  has been considered only in the following papers:   Devolder et al \cite{DGN2014jour}, 
Jiang et al. \cite{JiangSunToh2012jour}, Monteiro and Svaiter \cite{MonteiroSvaiter2013jour},
 Schmidt et al \cite{SchmidtLe-RouxBach2011prep},  and Villa et al.  \cite{VillaSalzoBaldassarreVerri2013} (the latter is the only work where $H$ is allowed to be infinite dimensional). In these works  \beqref{eq:x_kQ_k} is replaced by 
\begin{equation}\label{eq:x_kQ_kApprox}
x_k\approx \argmin_{x\in H}Q_k(x,y_k),
\end{equation}
where the approximation $\approx$ depends of  the perturbation terms and the notion of  inexactness \beqref{eq:x_kQ_kApprox} depends on the paper. 

In \cite{SchmidtLe-RouxBach2011prep} the inexactness \beqref{eq:x_kQ_kApprox} means
that $\tilde{Q}_k(x_k)\leq \epsilon_k+\tilde{Q}_k(y'_k)$ where $\epsilon_k>0$ is given, $y'_k$ is  a solution to an approximate quadratic minimization problem depending on previous iterations, and $\tilde{Q}_k$ is a perturbed  version of $Q_k$ obtained by perturbing the gradient of the quadratic term of $Q_k$ by a given error vector. 
In  \cite[p. 1046]{JiangSunToh2012jour} the authors consider a different approximation notion. Now \beqref{eq:x_kQ_kApprox} means that $F(x_k)\leq Q_k(x_k)+(\xi_k/(2t_k^2))$ and  $\|A_k^{-0.5}\delta_k\| \leq \epsilon_k/(\sqrt{2}t_k)$, where 
$\delta_k:=f'(y_k)+A_k(x_k-y_k)+\gamma_k$, $A_k:H\to H$ is some positive definite linear operator, $t_k>0$ is a parameter defined recursively (see \beqref{eq:t_k+1} below),  $\xi_k$ and $\epsilon_k$ are given positive parameters, 
and $\gamma_k\in \partial_{\epsilon_k/(2t_k^2)}(x_k)$. Here, as usual, for a given $\epsilon\geq 0$ the  $\epsilon$-subdifferential of $g$ is 
\begin{equation*}
\partial_{\epsilon}g(z)=\{u\in H: g(z)+\langle u,x-z\rangle\leq g(x)+\epsilon,\,\,\,\forall x\in H\}.
\end{equation*}
When $\epsilon_k=0$, then $\delta_k=0$ and \beqref{eq:x_kQ_k} is obtained.

The notion of inexactness of  \cite{VillaSalzoBaldassarreVerri2013} (see 
 \cite[Definition 2.1]{VillaSalzoBaldassarreVerri2013}, 
 \cite[Theorem 4.3]{VillaSalzoBaldassarreVerri2013})  is also related to the $\epsilon$-subdifferential: given an estimate parameter $\epsilon_k>0$, the approximation \beqref{eq:x_kQ_kApprox} holds if and only if $(y_k-\lambda_kf'(y_k)-x_k)/\lambda_k\in \partial_{\epsilon_k^2/(2\lambda_k)}(x_k)$, where $\lambda_k\in (0,2/L(f')]$ is a relaxation parameter. When $\epsilon_k=0$ then \beqref{eq:x_kQ_k} holds due to the optimality condition with $Q_k$. In \cite{MonteiroSvaiter2013jour} there is a discussion and general results which allow  inexactness, e.g., \cite[Sections 3-4]{MonteiroSvaiter2013jour}. However, the application of these results for the setting of a separable function, namely, \cite[Algorithm I]{MonteiroSvaiter2013jour}, is actually without inexactness. 

Finally,  in \cite{DGN2014jour} (see especially Definition  1 and the properties after it, Algorithm 3, and Subsection 8.2) this notion is related to the concept of an inexact first order oracle of a convex function called a $(\delta,L)$-oracle in \cite{DGN2014jour}. Here \beqref{eq:x_kQ_kApprox} means that $x_k=\argmin_{x\in C}\tilde{Q}_k(x,y_k)$ where $C$ is a fixed closed and convex subset of $H$ (the minimization is done over $C$ instead of over $H$) and  $\tilde{Q}_k$ is a quadratic upper bound on $F$  which coincides with $F$ and $y_k$. It is obtained from a modification of $Q_k$ by replacing in  $Q_k$ the coefficient $0.5L(f')$ of the  quadratic term $0.5L(f')\|x-y_k\|^2$ by $0.5L:=0.5(L(f')+(1/(2\delta))M^2)$. Here  $\delta:=\delta_k$ is an error term and $M>0$ is an upper bound on the variation of the subgradients of $g$ over $C$. Because one assumes that $M$ is finite, $C$ usually cannot be unbounded. As noted in \cite[p. 48]{DGN2014jour}, the parameter $\delta$ does not represent an actual accuracy and it can be chosen as small as one wants at the price of having a larger $L$, i.e, a worse quadratic upper bound on $F$.

\subsection{Contribution:} We consider two inexact versions (constant step size rule and  backtracking step size rule) of FISTA and show that FISTA is  perturbation resilient in the function values, namely, it still converges non-asymptotically despite a certain type of perturbations which appear in the algorithmic sequences. The notion of inexactness we consider  is of the form 
\begin{equation*}
x_k=e_k+\argmin_{x\in H}Q_k(x,y_k), 
\end{equation*}
which seems to be rather simple comparing to notions considered in previous mentioned works. Such a notion of inexactness is closely related to notions considered by, for instance, Combettes-Wajs \cite[Theorem 3.4]{CombettesWajs2005jour},  Rockafellar  
\cite[Theorem 1]{Rockafellar1976jour},  and Zaslavski \cite[Theorem 1.2]{Zaslavski2010b-jour} in a different context (non-accelerated proximal methods). Depending on the rate of decay of the magnitude of the perturbations $e_k$ to zero, either the original $O(1/k^2)$ convergence rate is preserved or a slower one is obtained. Interestingly, despite the difference in the notion of inexactness and in the algorithmic schemes, the rate of decay we obtain is closely related to other schemes 
(Corollaries \bref{cor:F=O()}-\bref{cor:Decay e_k} and Remark \bref{rem:DecayLiterature} below). 
We allow the ambient space $H$ to be infinite dimensional, as in  \cite{VillaSalzoBaldassarreVerri2013} (and  \cite{SalzoVilla2012jour,Tseng2008prep}, which, however, do not consider inexactness) but not elsewhere. Unless the perturbations vanish, we require $g$ to be finite. This is a somewhat stronger condition than in several previous works in which $g$ was  allowed to attain the value $\infty$, but it is more general than the original paper of Beck and Teboulle \cite{BeckTeboulle2009jour}. In contrast to previous works on inexact accelerated methods, we allow the case $\inf_{x\in H} F(x)=-\infty$ 
and we do not require the optimal set to be  nonempty (for the exact case, only  \cite{SalzoVilla2012jour,Tseng2008prep,Tseng2010jour,VillaSalzoBaldassarreVerri2013} allow this latter case).

Our analysis is motivated by \cite{BeckTeboulle2009jour}, but a few significant differences exist, partly because of the presence of perturbation terms and the infinite dimensional setting.  An interesting by-product of our analysis  is the derivation, in a systematic way, of the parameters involved in FISTA, parameters whose source seems to be a mystery. For instance, in all previous works which discuss accelerated proximal gradient methods, one  uses the auxiliary variable $y_k$  and assumes an explicit linear dependence of it on previous iterations (see \beqref{eq:x_kQ_k} above and the discussion after it). 
The variable $y_k$ is assumed to depend on positive parameters $t_k,t_{k-1},\ldots$ which satisfy a certain relation (e.g., \beqref{eq:t_k+1} below), but no systematic method is presented which explains the  mysterious origin of both $y_k$ and $t_k$: it seems that initially they were guessed, and in later works they or slight variants of them were used directly without shedding light on their origin. In our analysis we do not impose in advance any form on $y_k$ or $t_k$, but rather derive them  explicitly during the proof (until late  stages in the proof we only require the  existence of $y_k$ satisfying \beqref{eq:x_kQ_k} without any relation to $t_k$ whose existence is even not assumed). After deriving our ideas, we have become aware of the works of Tseng \cite[Proof of Proposition 2]{Tseng2008prep}, \cite[Proof of Theorem 1(b)]{Tseng2010jour} which also shed some light on the  origin of $t_k$ and $y_k$ (in the exact case). However, his analysis  is different (but not entirely different) from ours. 

As said in Subsection \bref{subsec:Background}, the superiorization methodology is one of the reasons to consider the question of perturbation resilience in the context of FISTA. Our final  contribution in this paper is to re-examine this methodology in a comprehensive way and to  significantly extend its  scope. 

\subsection{Paper layout:} Basic assumptions and the formulation of inexact versions of FISTA are given in Section \bref{sec:FISTA-Perturbed}. The convergence of the iterative schemes are presented in Section \bref{sec:Converge}, as well as several corollaries and remarks related to  the convergence theorem (mainly regarding the rate of decay of the error terms and the function values), including a comparison with related papers. The superirization methodology is re-examined and extended in Section \bref{sec:Superiorization}. The proofs of some  auxiliary claims  are given in the appendix (Section \bref{sec:Appendix}). 

\section{Basic assumptions and the formulation of FISTA with perturbations}\label{sec:FISTA-Perturbed}

\subsection{Basic assumptions:} 
From now on $H$ is a given real Hilbert space with an inner product $\langle\cdot,\cdot\rangle$ and an induced norm $\|\cdot\|$. We define $F:H\to (-\infty,\infty]$ by $F:=f+g$ where $f:H\to\R$ is a given convex function whose derivative $f'$ exists and is Lipschitz continuous with a  Lipschitz constant $L(f')\geq 0$, i.e., $\|f'(x)-f'(y)\|_*\leq L\|x-y\|$ for all $x,y\in H$ where $\|\cdot\|_*$ is the norm of the dual $H^*$ of $H$. We assume that $g$ is a given convex and lower semicontinuous function from $H$ to $(-\infty,\infty]$ which is also proper, i.e., its effective  domain $\dom(g):=\{x\in H: g(x)<\infty\}$ is nonempty. 

\subsection{The definition of the accelerated scheme:} The scheme has two versions: a constant  step size version and  a backtracking version. The constant step size version with perturbations  is defined as follows:\\

\noindent{\bf Input:} a positive number $L\geq L(f')$.\\
\noindent{\bf Step 1 (initialization):} arbitrary $x_1\in H$, $y_2\in H$, $t_2\geq 1$.\\
\noindent{\bf Step $k$, $k\geq 2$: } Let $L_k:=L$
\begin{equation}\label{eq:x_kConstStep}
x_k=p_{L_k}(y_k)+e_k, 
\end{equation}
where $e_k\in H$ is the error term,
\begin{equation}\label{eq:p_L}
p_{L_k}(y):=\textnormal{argmin}\{x\in H:\,\, Q_{L_k}(x,y)\}, 
\end{equation}
\begin{equation}\label{eq:Q_L}
Q_{L_k}(x,y):=f(y)+\langle f'(y), x-y\rangle+0.5L_k\|x-y\|^2+g(x), 
\end{equation} 
\begin{equation}\label{eq:y_k+1ConstStep}
y_{k+1}=x_k+\frac{t_k-1}{t_{k+1}}\Big(x_k-x_{k-1}\Big),  
\end{equation}
\begin{equation}\label{eq:t_k+1}
t_{k+1}=\frac{1+\sqrt{1+4t_{k}^2}}{2}.  
\end{equation}

The backtracking step size version with perturbations is defined as follows: \\
\noindent{\bf Input:  $L_1>0$, $\eta>1$}.\\
\noindent{\bf Step 1 (initialization)} arbitrary $x_1\in H$, $y_2\in H$, $t_2\geq 1$.\\
\noindent{\bf Step $k$, $k\geq 2$: } Find the smallest nonnegative integer $i_k$ such 
that with $L_k:=\eta^{i_k}L_{k-1}$ we have 
\begin{equation}\label{eq:L_k_FISTA}
F(p_{L_k}(y_{k}))\leq Q_{L_k}(p_{L_k}(y_k),y_k).
\end{equation}
 Now let 
\begin{equation}\label{eq:x_kBacktrack}
x_k=p_{L_k}(y_k)+e_k, 
\end{equation}
\begin{equation}\label{eq:y_k+1BackTrack}
y_{k+1}=x_k+\frac{t_k-1}{t_{k+1}}\Big(x_k-x_{k-1}\Big)
\end{equation}
where $t_{k+1}$ is defined in \beqref{eq:t_k+1}. 
In both versions the error terms $e_k$ are arbitrary vectors in $H$ satisfying a certain adaptivity condition which is presented later (Subsection \bref{subsec:ConditionErrorTerms} below) and depends on the 
boundedness of $F$ on a certain ball with center $x_k$. As is well-known, the minimizer of $x\mapsto Q_{L_k}(x,y_k)$ exists and is unique \cite[Corollary 11.15]{BauschkeCombettes2011book}. Thus $p_{L_k}(y_k)$ and $x_k$ are well defined.

\begin{remark}\label{rem:L_kFiniteFx_k}
We note that the backtracking step size rule is well-defined because 
according to a well-known finite dimensional result \cite[Lemma 1.2.3, pp. 22-23]{Nesterov2004book}, whose proof in the infinite dimensional case is similar (see Lemma \bref{lem:LipschitzUpperBound} in the appendix), 
if $f:H\to \R$ is continuously differentiable with a Lipschitz constant $L(f')$ of $f'$, then 
\begin{equation}\label{eq:fL}
f(x)\leq f(y)+\langle f'(y), x-y\rangle+0.5L\|x-y\|^2,\quad \forall\, x,y\in H,\, \forall\, L\geq L(f'). 
\end{equation}
By adding $g(x)$ to both sides of \beqref{eq:fL} and using the representation $F=f+g$, we conclude that $F(p_L(y_k))\leq Q_L(p_L(y_k),y_k)$. Since 
for large enough $i_k$ (and, obviously, also in the constant step size rule) we will have $L_k\geq L(f')$, the above implies $F(p_{L_k}(y_k))\leq Q_{L_k}(p_{L_k}(y_k),y_k)$. However,   it may happen that $F(p_{L_k}(y_k))\leq Q_{L_k}(p_{L_k}(y_k),y_k)$ even when $L_k<L(f')$. 

In addition, the  minimization in \beqref{eq:p_L} can be done over the effective domain of $g$. It follows that $Q_L(p_{L_k}(y_k),y_k)$ is always finite for all $k\geq 2$. Therefore,  if $e_k=0$, then $F(x_k)$ is finite because in this case the argmin in \beqref{eq:p_L} is attained at $x_k=p_{L_k}(y_k)$ and from the previous paragraph we  have  $F(x_k)\leq Q_{L_k}(x_k,y_k)$ for all  $k\geq 2$. 
\end{remark}

\begin{remark}
There is a certain delicate point regarding the backtracking step size version: in many cases  computing both sides in $F(p_{L_k}(y_k))\leq Q_{L_k}(p_{L_k}(y_k),y_k)$ 
is not accurate because $p_{L_k}(y_k)$ is known only up to an error $e_k$, namely, one 
actually is able to compute only $x_k$. Thus, unless we have an exact expression for $p_{L_k}(y_k)$,  we actually check whether $F(x_k)\leq Q_{L_k}(x_k,y_k)$. 
The $L_k$ for which $F(x_k)\leq Q_{L_k}(x_k,y_k)$ holds may not satisfy 
$F(p_{L_k}(y_k))\leq Q_{L_k}(p_{L_k}(y_k),y_k)$. So we need to find a simple condition which ensures that if $F(x_k)\leq Q_{L'_k}(x_k,y_k)$ holds for some $L'_k$, then $F(p_{L_k}(y_k))\leq Q_{L_k}(p_{L_k}(y_k),y_k)$ for some explicit positive number $L_k$. 
In the constant step version there is no problem assuming  we can evaluate $L(f')$ from above, since in this case we can take $L_k$ to be any positive upper bound on $L(f')$. The problem is 
with the backtracking step size version, unless $e_k=0$. 
\end{remark}

\begin{remark}
 The construction of $L_k$ and \beqref{eq:fL} imply that 
\begin{equation}\label{eq:tau_rho}
\rho\leq L_k\leq \tau,\quad \forall\, k\geq 1 
\end{equation}
for some positive numbers $\rho\leq \tau$. Indeed, if $L_1<L(f')$,  
then \beqref{eq:tau_rho} holds with $\rho:=L_1$ and $\tau:=\eta L(f')$. 
If $L_1\geq L(f')$, then \beqref{eq:tau_rho} holds with $\rho:=L_1=:\tau$. 
\end{remark}

\subsection{The condition on the error terms}\label{subsec:ConditionErrorTerms}

In the following lines a condition on the error terms will be presented, namely \beqref{eq:e_k}.  For a variation of this condition, see Remark \bref{rem:sigma_k'} below. 
Let $\tilde{x}\in H$ be such that $F(\tilde{x})<\infty$ (there exists such an $\tilde{x}$ since $F$ is proper) and let $s_1>0$ be fixed (for all $k$). Let $\mu>0$ be a fixed upper bound on $\|\tilde{x}\|$.  
Let $(s_k)_{k=2}^{\infty}$ a sequence of arbitrary nonnegative numbers. Denote by $B[x_k,2s_1]$  the closed ball of radius $2s_1$ and center $x_k$. 
If $F$ is bounded on $B[x_k,2s_1]$, then let $m_{k}$ and $M_{k}$ be any lower and upper bounds of $F$ on $B[x_k,2s_1]$, respectively, satisfying $m_k<M_k$. Define 
\begin{equation}\label{eq:Lambda_{k}}
\Lambda_{k}:=\left\{
\begin{array}{lll}
\displaystyle{\frac{M_{k}-m_{k}}{s_1}}, & \textnormal{if}\,\,F\,\,\textnormal{is bounded on}\,\, B[x_k,2s_1],\\
0 & \textnormal{otherwise}.
\end{array} 
\right.
\end{equation}
For all $k\geq 2$, let 
\begin{equation}\label{eq:sigma_k}
\sigma_k:=2t_{k}^2\left((1/L_{k})\Lambda_{k}+\|p_{L_{k}}(y_{k})\|
+\|p_{L_{k-1}}(y_{k-1})\|+4s_1+(1/t_{k})\mu\right),
\end{equation}
where $p_{L_1}(y_1):=x_1$. Then for all $k\geq 2$ the error term $e_k$ is any vector in $H$ which satisfies the following condition: 
\begin{equation}\label{eq:e_k}.
\|e_{k}\|\leq\left\{
\begin{array}{lll}
\min\left\{s_1,s_k/\sigma_k\right\},& 
\textnormal{if}\,F\,\textnormal{is bounded on}\,\,B[x_k,2s_1],\\
0 & \textnormal{otherwise}.
\end{array}\right.
\end{equation}
Here are a few remarks regarding Condition \beqref{eq:e_k}. 
\begin{remark}
First, if  $F$ is  bounded  on the  considered ball, then $g$ must be bounded there because $f$ is always bounded on balls (follows from the fact that $f'$ is Lipschitz continuous). When $H$ is finite dimensional and $g$ is continuous (as happens  in many applications: see e.g., Section \bref{sec:Intro}), then the boundedness of $g$ is automatically ensured because closed balls are compact so classical theorems in analysis can be used. If however $g$ attains the value $\infty$, as happens when $g$ is an indicator function of a closed and  convex subset, then we must require the error term $e_k$ to vanish if $H$ is finite or infinite dimensional. In the infinite dimensional case there is another complication, since then there are exotic cases \cite[Example 7.11, p. 413]{bb1996} in which $g$ may be unbounded on closed balls even if it is continuous and does not  attain the value $\infty$. However, for most  applications (e.g., the infinite dimensional versions of the examples given in Section  \bref{sec:Intro}) this does not happen.  
\end{remark}

\begin{remark} 
Condition \beqref{eq:e_k} implies the dependence of $e_k$ on previous iterations. Hence this condition can be regarded as being adaptive or relative. Conditions in this spirit have been dealt with in the literature in \cite{CensorReem2015jour}. 
In \cite{Eckstein1998jour,OteroIusem2012,IusemOtero2001jour,IusemPennanenSvaiter2003} one can 
find related but more implicit relations. In other places, e.g., \cite{Combettes2004jour,CombettesWajs2005jour} 
there is no such dependence, but rather the error terms should be summable. 
In previous works dealing with inexact accelerated methods in the context of separable functions   \cite{DGN2014jour,JiangSunToh2012jour,SchmidtLe-RouxBach2011prep,VillaSalzoBaldassarreVerri2013}  the error terms are assumed to decay fast enough to zero by imposing a pure numerical quantity  which bounds their magnitude (or the sum of their magnitudes) from above. 
\end{remark}
 
\begin{remark}\label{rem:sigma_k'}

It can be argued that \beqref{eq:e_k} is not explicit enough for two reasons. First, 
$\tilde{x}$ is sometimes a minimizer (see the formulation of Theorem \bref{thm:FISTA-Perturb}  below), so it is not known, and hence its upper bound $\mu$ is unknown. Second, unless it is known that $e_k=0$, we usually cannot compute $p_{L_k}(y_k)$, hence we do not know it, but instead  we know $x_k$. Therefore it is a problem to compute $\sigma_k$ and to estimate $\|e_k\|$. 

Here is  an answer to the first concern (see the paragraphs below for the second concern). If $\tilde{x}$ is a minimizer, then 
it is indeed unknown. However, $\|\tilde{x}\|$ can be estimated frequently. Assume 
for instance that $F$ is coercive, i.e., $\lim_{\|x\|\to\infty}F(x)=\infty$. 
In particular, there is some $\mu>0$ such that $F(x)>F(0)$ for all $\|x\|>\mu$. Since 
$\tilde{x}$ is a minimizer of $F$ we have $F(\tilde{x})\leq F(0)$, and so it must be that $\|\tilde{x}\|\leq \mu$. If for example $F(x)=\|Ax-b\|^2+\|x\|_1$, $x\in H:=\R^n$, $A:\R^n\to \R^{n'}$ is linear, $n,n'\in\N$, $b\in \R^{n'}$, then $F(x)\geq \|x\|_1\geq \|x\|$ for all $x\in H$. As a result, we obtain that $F(x)>F(0)=\|b\|^2$ whenever $\|x\|>\mu:=\|b\|^2$. 

As for the second concern, we suggest three ways to overcome the problem. 
First, it is worth noting that there are 
important situations in which $p_{L_k}(y_k)$ can be computed exactly: one of them is the $\ell_1-\ell_2$ optimization  case, namely, when $H=\R^n$, $f(x)=\|Ax-b\|^2$, $A:\R^n\to\R^{n'}$ is linear with adjoint $A^*:\R^{n'}\to\R^n$, $b\in \R^{n'}$, $g(x)=\lambda\|x\|_1$, since then $Q_{L_k}(x,y_k)=f(y_k)+\langle f'(y_k),x-y_k\rangle+0.5L_k\|x-y_k\|^2+\lambda\|x\|_1$. In this case one has 
$p_{L_k}(y_k)={S}_{\lambda/L_k}(y_k-(2/L_k)A^*(Ay_k-b))$, where ${S}_{\alpha}:\R^n\to\R^n$ 
is the shrinkable operator which maps the vector $x$ to the vector  $S_{\alpha}(x)=(\max\{|x_i|-\alpha,0\}\sign(x_i))_{i=1}^n$ for each  given $\alpha>0$. See, for instance, \cite[pp. 185,188]{BeckTeboulle2009jour},\cite[p. 80, Equation  (21)]{ZibulevskyElad2010jour}. In such cases it is also possible to use the perturbations in an active way, e.g., as a mean for enhancing the speed of convergence or for achieving other purposes, as done in the superiorization scheme (see Section \bref{sec:Superiorization} below).

A second way to overcome the second concern can be used in the frequent case where $p_{L_k}(y_k)$ can be computed only approximately. In this case, given $k\geq 2$ we can replace \beqref{eq:e_k} by the following condition:  
\begin{equation}\label{eq:e_k'}
\|e_{k}\|\leq\left\{
\begin{array}{lll} \min\left\{s_1,s_k/\sigma'_k\right\},& 
\textnormal{if}\,F\,\textnormal{is bounded on}\,\,B[x_k,2s_1],\\
0 & \textnormal{otherwise}.
\end{array}
\right.
\end{equation}
where 
\begin{equation}\label{eq:sigma_k'}
\sigma'_k:=2t_{k}^2\left((1/L_{k})\Lambda_{k}+
\|x_k-((t_{k}-1)/t_{k})x_{k-1}\|+2s_1+(1/t_{k})\mu\right),
\end{equation}
and our convergence results still remain correct due to \beqref{eq:BoundEpsilon_k+1} below. 
For applying \beqref{eq:e_k'} in practice we first approximate $p_{L_k}(y_k)$ 
 up to some arbitrary small parameter $\epsilon>0$. Some examples are 
 mentioned in \cite{VillaSalzoBaldassarreVerri2013} 
 (this follows from   \cite[Proposition 2.5]{SalzoVilla2012jour} and the discussion after  \cite[Definition  2.1]{SalzoVilla2012jour}); see also some of the examples in  \cite{BeckTeboulle2009jour,JiangSunToh2012jour}. What we obtain is a point 
 $x_k$ for which we know that  $\|x_k-p_{L_k}(y_k)\|\leq \epsilon$. Now we check whether 
$\epsilon\leq \min\left\{s_1,s_k/\sigma'_k\right\}$. If yes, then for sure $e_k:=x_k-p_{L_k}(y_k)$ satisfies \beqref{eq:e_k'}. Otherwise, we continue to approximate 
$p_{L_k}(y_k)$ using a smaller parameter, say $0.5\epsilon$, and calling it again $\epsilon$ (of course, $k$ is fixed during this process).  Eventually the inequality 
$\epsilon\leq \min\left\{s_1,s_k/\sigma'_k\right\}$ will be satisfied since 
$\sigma'_k\geq 4t_{k-1}^2s_1>0$. 

The third way to overcome the second concern is to bound from above $\sigma_k$ by some explicit parameter $\tilde{\sigma}_k$. Then we can take any vector $e_k\in H$ satisfying 
\begin{equation}\label{eq:e_kTilde_sigma_k}.
\|e_{k}\|\leq\left\{
\begin{array}{lll}
\min\left\{s_1,s_k/\tilde{\sigma}_k\right\},& 
\textnormal{if}\,F\,\textnormal{is bounded on}\,\,B[x_k,2s_1],\\
0 & \textnormal{otherwise}.
\end{array}\right.
\end{equation}
Such $e_k$ will satisfy \beqref{eq:e_k} too. It remains to estimate $\sigma_k$ from above. 
As follows from \cite[Proposition 11.13, p. 158]{BauschkeCombettes2011book}, the function $u:H\to(-\infty,\infty]$ defined  by  $u(x):=Q_{L_k}(x,y_k)$ for each $x\in H$ is supercoercive (i.e., $\lim_{x\to\infty}u(x)/\|x\|=\infty$) because it is a sum of a quadratic (hence supercoercive) function and a convex, proper and lower  semicontinuous function. Consequently  there exists $\nu_k$ large enough such that for all $x$ satisfying $\|x\|>\nu_k$ we have, in particular, that $u(x)>u(0)$. Since $p_{L_k}(y_k)$ is the minimizer of $u$ we conclude that $u(x)>u(0)\geq u(p_{L_k}(y_k))$ for each $x$ which satisfies $\|x\|>\nu_k$. Therefore $\|p_{L_k}(y_k)\|\leq \nu_k$. Thus $\sigma_k$ is bounded from above by 
\begin{equation}\label{eq:Tilde_sigma_k}
\tilde{\sigma}_k:=2t_{k}^2\left((1/L_{k})\Lambda_{k}+\nu_k
+\nu_{k-1}+4s_1+(1/t_{k})\mu\right).
\end{equation}
\end{remark}

\section{The convergence theorem}\label{sec:Converge}
The proof of the main convergence theorem (Theorem \bref{thm:FISTA-Perturb} below) is based on  several lemmas. The first one is a generalization of \cite[Lemma 2.3]{BeckTeboulle2009jour} to the case where $H$ is infinite dimensional and $g$ is 
lower semicontinuous. A large part of the proof is similar to  \cite[Lemma 2.3]{BeckTeboulle2009jour} and hence we decided to put it in the appendix.

\begin{lem}\label{lem:FB_k}
Suppose that $y\in H$ and $L>0$ satisfy $F(p_L(y))\leq Q_L(p_L(y),y)$ where $Q_L$ is defined in \beqref{eq:Q_L} with $L$ instead of $L_k$. 
Then for all $x\in H$ 
\begin{equation}\label{eq:FB_k}
F(x)-F(p_L(y))\geq 0.5L\|p_L(y)-y\|^2+L\langle p_L(y)-y,y-x\rangle.
\end{equation}
\end{lem}

\begin{remark}
The definition of $L_k$ and Remark \bref{rem:L_kFiniteFx_k} imply that we can use Lemma \bref{lem:FB_k} with $y=y_k$, $L=L_k$, and   an arbitrary $x\in H$. 
\end{remark}
The next lemma is perhaps known and its proof is given for the sake of completeness. 
\begin{lem}\label{lem:LipBounded}
Let $(X,\|\cdot\|)$ be a real normed space and let $G:X\to (-\infty,\infty]$ be convex. Let $B\subset X$ be a closed ball with radius $r_B$ and center $a\in X$ and let $B'$ be any closed  ball containing $B$ with the same center $a$ and with a radius $r_{B'}>r_B$. Suppose that 
there exist real numbers $m_{B'}\leq M_{B'}$ such that $m_{B'}\leq G(x)\leq M_{B'}$ for  all $x\in B'$. Then $G$ is Lipschitz on $B$ with a Lipschitz constant 
\begin{equation}\label{eq:Lambda}
\Lambda:=(M_{B'}-m_{B'})/(r_{B'}-r_B). 
\end{equation}
\end{lem}

\begin{proof}
The proof is closely related to the proof of \cite[Theorem 2.21, p. 69]{VanTiel1984}. Let $x,y\in B$ be arbitrary. Denote $r:=r_B$, $r':=r_{B'}>r$. 
Let $z:=y+((r'-r)/\|y-x\|)(y-x)$ if $x\neq y$ and $z:=y=x$ otherwise. Then $y=\lambda z+(1-\lambda)x$ where $\lambda:=\|x-y\|/(\|x-y\|+r'-r)$. Since $\lambda\in [0,1]$, the convexity of $G$ implies that  $G(y)\leq \lambda G(z)+(1-\lambda)G(x)$. Since $y\in B$, the definition of $z$ implies that  $\|z-a\|\leq \|y-a\|+r'-r\leq r'$. Therefore $z\in B'$. The above inequalities, 
the definition of $\lambda$, and the fact that $G(x),G(y)\in\R$ (since $x,y\in B'$) imply the inequality 
\begin{equation}
G(y)-G(x)\leq \lambda(G(z)-G(x))\leq \lambda (M_{B'}-m_{B'})\leq \frac{(M_{B'}-m_{B'})\|x-y\|}{r'-r}.
\end{equation} 
By interchanging the role of $x$ and $y$ we obtain  
\begin{equation}
G(x)-G(y)\leq \frac{(M_{B'}-m_{B'})\|y-x\|}{r'-r}.
\end{equation}
Since $x$ and $y$ were arbitrary points in $B$, it follows that  $G$ is Lipschitz on $B$ 
with a Lipschitz constant given in \beqref{eq:Lambda}, as claimed. 
\end{proof}

\begin{remark}
As shown in \cite[Proposition 7.8]{bb1996} for the case where $X$ is Hilbert, boundedness of $G$ on balls (hence on bounded subsets) is equivalent to $G$ being Lipschitz on bounded subsets and also equivalent to the existence and uniform boundedness of the subgradients of $G$ on bounded subsets. In the finite dimensional case all of these conditions always hold as a corollary of \cite[Theorem 5.23, p. 70]{VanTiel1984}, but the counterexample given in 
\cite[Example 7.11, p. 413]{bb1996} shows that in the infinite dimensional case they do not necessary hold. 
\end{remark}
In order to formulate Theorem \bref{thm:FISTA-Perturb} below, we need the following definition. 
\begin{defin}
$F$ is said to be double bounded if it is bounded on bounded subsets of $H$.
\end{defin}

\begin{thm}\label{thm:FISTA-Perturb}
In the framework of Section \bref{sec:FISTA-Perturbed}, suppose that one of the following two  possibilities hold: either we are in the backtracking step size rule, and then the optimal set of $F$ is nonempty and we fix an arbitrary minimizer $\tilde{x}$ in the optimal set, or we are in the constant step size rule, and then we fix an arbitrary $\tilde{x}\in H$ for which $F(\tilde{x})$ is finite. 
Then for all $k\geq 1$ 
\begin{equation}\label{eq:BoundFunctionValues}
F(x_{k+1})-F(\tilde{x})\leq \frac{2\tau\left((2/L_1)t_1(t_1-1)(F(x_1)-F(\tilde{x}))+\|t_1y_2-(t_1-1)x_1-\tilde{x}\|^2+\sum_{j=2}^{k+1}s_{j}\right)}{(k+1)^2}.
\end{equation}
If, in addition, 
\begin{equation}\label{eq:LimSum0}
\lim_{k\to\infty}\frac{\sum_{j=2}^{k+1}s_j}{(k+1)^2}=0,
\end{equation}
then $\lim_{k\to\infty}F(x_k)=\inf_H F$. In particular, the above holds when  
$F$ is double bounded and also when $F$ is not double bounded but 
 $e_k=0$ for all $k\geq 2$.  
\end{thm}

\begin{proof}
During the proof all the relevant expressions will be derived. In particular, there will be no use of the specific form of $y_{k}$ until \beqref{eq:y_kExplicit} (only the existence of $y_k\in H$ which satisfies \beqref{eq:x_kConstStep} or \beqref{eq:x_kBacktrack} will be assumed) and  no use of of the specific form of $t_{k+1}$ until  \beqref{eq:t_k=delta_k+1} (the existence of 
$t_k$ will not even be assumed until  \beqref{eq:t_k=delta_k+1}). For the sake of convenience, the proof is divided into several steps. \\

{\bf \noindent Step 1:} Fix an arbitrary $k\geq 1$.  
Let $B_{k+1}:=B[x_{k+1},s_1]$, $B'_{k+1}:=B[x_{k+1},2s_1]$ and $v_{k}:=F(x_{k})-F(\tilde{x})$. 
Either $F$ is bounded on $B'_{k}$ and then $F(x_{k})$ is finite, or $F$ is not bounded there and then $e_{k}=0$ from \beqref{eq:e_k} (or \beqref{eq:e_k'}). This, together with Remark  \bref{rem:L_kFiniteFx_k}, implies that if in addition $k\geq 2$, then also in this case $F(x_{k})$ is finite. Since we  always assume that  $F(\tilde{x})$ is finite it follows that $v_{k+1}$ is finite for all $k\in\N$. 

From now until the last paragraph of this step (excluding) assume that $F$ is bounded on $B'_{k+1}$. At the end of the step we will deal with the second possibility. Let $m_{k+1}$ and $M_{k+1}$, $m_{k+1}<M_{k+1}$ be any lower and upper bounds of $F$ on $B'_{k+1}$, respectively,  and let  $\Lambda_{k+1}:=(M_{k+1}-m_{k+1})/s_1$. 
Since $\|e_{k+1}\|\leq s_1<2s_1$ and since Lemma \bref{lem:LipBounded} implies that  $F$ is Lipschitz on $B'_{k+1}$ with a Lipschitz constant $\Lambda_{k+1}$, 
we have 
\begin{equation}\label{eq:Lambda_e}
\Lambda_{k+1}\|e_{k+1}\|-F(x_{k+1})\geq -F(x_{k+1}-e_{k+1}). 
\end{equation}
Thus, 
by substituting $x=x_k$, $y=y_{k+1}$, $L=L_{k+1}$ in \beqref{eq:FB_k},  
 using the definition of $v_k$, using \beqref{eq:x_kBacktrack} (or \beqref{eq:x_kConstStep}), Lemma \bref{lem:FB_k}, and using \beqref{eq:Lambda_e},  we obtain 
\begin{multline}\label{eq:v_kv_k+1}
2(v_k+\Lambda_{k+1}\|e_{k+1}\|-v_{k+1})/L_{k+1}\geq 2(F(x_k)-F(x_{k+1}-e_{k+1}))/L_{k+1}\\
\geq\|x_{k+1}-e_{k+1}-y_{k+1}\|^2+2\langle x_{k+1}-e_{k+1}-y_{k+1},y_{k+1}-x_k\rangle.
\end{multline}
By substituting  $x=\tilde{x}$, $y=y_{k+1}$, $L=L_{k+1}$ in \beqref{eq:FB_k} and using \beqref{eq:x_kBacktrack} (or \beqref{eq:x_kConstStep}), Lemma \bref{lem:FB_k} and  \beqref{eq:Lambda_e}, we obtain 
\begin{multline}\label{eq:-v_k+1}
2(\Lambda_{k+1}\|e_{k+1}\|-v_{k+1})/L_{k+1}\\
\geq 2(F(\tilde{x})-F(x_{k+1}-e_{k+1}))/L_{k+1}\geq\|x_{k+1}-e_{k+1}-y_{k+1}\|^2+2\langle x_{k+1}-e_{k+1}-y_{k+1},y_{k+1}-\tilde{x}\rangle.
\end{multline}
Now we multiply \beqref{eq:v_kv_k+1} by a nonnegative  number $\gamma_k$ (to be determined 
later) and add the resulting inequality to \beqref{eq:-v_k+1}. We have 
\begin{multline}\label{eq:gamma_k}
(2/L_{k+1})(\gamma_kv_k-(1+\gamma_k)v_{k+1})+(2/L_{k+1})(1+\gamma_k)\Lambda_{k+1}\|e_{k+1}\|\\
\geq (\gamma_k+1)\|x_{k+1}-e_{k+1}-y_{k+1}\|^2+2\langle x_{k+1}-e_{k+1}-y_{k+1},
\gamma_k(y_{k+1}-x_k)+y_{k+1}-\tilde{x}\rangle\\
=(\gamma_k+1)\|x_{k+1}-y_{k+1}\|^2+(\gamma_k+1)\|e_{k+1}\|^2
-2\langle e_{k+1},(\gamma_k+1)(x_{k+1}-y_{k+1})\rangle\\
+2\langle x_{k+1}-y_{k+1},\gamma_k(y_{k+1}-x_k)+y_{k+1}-\tilde{x}\rangle
-2\langle e_{k+1},\gamma_k(y_{k+1}-x_k)+y_{k+1}-\tilde{x}\rangle.
\end{multline}
Now we multiply \beqref{eq:gamma_k} by some nonnegative number $\delta_k$ (to be determined later).  We have 
\begin{multline}\label{eq:left_right}
(2/L_{k+1})(\delta_k\gamma_kv_k-\delta_k(1+\gamma_k)v_{k+1})
+(2/L_{k+1})\delta_k(1+\gamma_k)\Lambda_{k+1}\|e_{k+1}\|\\
\geq \|(\delta_k(1+\gamma_k))^{0.5}(x_{k+1}-y_{k+1})\|^2+2\delta_k\left\langle x_{k+1}-y_{k+1},(1+\gamma_k)y_{k+1}-(\gamma_kx_k+\tilde{x})\right\rangle\\
+\delta_k(1+\gamma_k)\|e_{k+1}\|^2-2\delta_k\langle e_{k+1}, (1+\gamma_k)x_{k+1}-(\gamma_k x_k+\tilde{x})\rangle.
\end{multline}
So far we assumed that $F$ is bounded on $B'_{k+1}$. However, if it is not bounded there, then 
according to \beqref{eq:e_k} or \beqref{eq:e_k'} we have $e_{k+1}=0$, and then \beqref{eq:Lambda_e}-\beqref{eq:left_right} still hold (with arbitrary $\Lambda_{k+1}\in\R$), again because of Lemma \bref{lem:FB_k} and the same simple algebra. The above inequalities also hold, trivially, when $v_k=\infty$ (can happen only when $k=1$). \\

{\bf \noindent Step 2:} 
In order to reach useful expressions, we want to use the simple vectorial identity 
\begin{equation}\label{eq:VectorialIdentity}
\|b-a\|^2+2\langle b-a,a-c\rangle=\|b-c\|^2-\|a-c\|^2,
\end{equation}
which seems related to the right hand side of \beqref{eq:left_right} (if we ignore for a moment the terms involving perturbations). In order to use it, we impose additional assumptions on the sequences $(\gamma_k)_{k=1}^{\infty}$ 
and $(\delta_k)_{k=1}^{\infty}$ (in addition to non-negativity): 
\begin{equation}\label{eq:1+gamma_k=delta_k}
1+\gamma_k=(\delta_k(1+\gamma_k))^{0.5}=\delta_k,\quad \forall k\geq 1.
\end{equation}
Fortunately, these three equations are consistent  and once we assume \beqref{eq:1+gamma_k=delta_k}, substitute 
\begin{equation}
a=\delta_k y_{k+1},\quad b=\delta_k x_{k+1},\quad 
c=\gamma_k x_k+\tilde{x}
\end{equation}
in \beqref{eq:VectorialIdentity}, use the Cauchy-Schwarz inequality, and use \beqref{eq:left_right}, we obtain 
\begin{multline}\label{eq:abc}
(2/L_{k+1})(\delta_k(\delta_k-1)v_k-\delta_k^2v_{k+1})
-\delta_k^2\|e_{k+1}\|^2\\
+(2/L_{k+1})\delta_k^2\Lambda_{k+1}\|e_{k+1}\|
+2\delta_k^2\|e_{k+1}\|\|x_{k+1}-((\delta_k-1)/\delta_k)x_k-\tilde{x}/\delta_k\|\\
\geq \|\delta_k x_{k+1}-(\gamma_k x_k+\tilde{x})\|^2-\|\delta_k y_{k+1}-(\gamma_k x_k+\tilde{x})\|^2.
\end{multline}
With the notation
\begin{equation}\label{eq:epsilon_k+1}
\epsilon_{k+1}:=2\delta_k^2\|e_{k+1}\|((\Lambda_{k+1}/L_{k+1})
+\|x_{k+1}-((\delta_k-1)/\delta_k)x_k-\tilde{x}/\delta_k\|)
\end{equation}
and the fact that $-\delta_k^2\|e_{k+1}\|^2\leq 0$ we obtain the inequality 
\begin{equation}\label{eq:abc-epsilon}
(2/L_{k+1})(\delta_k(\delta_k-1)v_k-\delta_k^2v_{k+1})
+\epsilon_{k+1}
\geq \|\delta_k x_{k+1}-(\gamma_k x_k+\tilde{x})\|^2-\|\delta_k y_{k+1}-(\gamma_k x_k+\tilde{x})\|^2.
\end{equation}
Now there are two possibilities: if we are in the constant step size rule, then  $L_{k+1}=L_k$ and we obtain from  \beqref{eq:abc-epsilon} that 
\begin{equation} \label{eq:almost_a_k}
(2/L_{k})\delta_k(\delta_k-1)v_k-(2/L_{k+1})\delta_k^2v_{k+1}
+\epsilon_{k+1}
\geq \|\delta_k x_{k+1}-(\gamma_k x_k+\tilde{x})\|^2-\|\delta_k y_{k+1}-(\gamma_k x_k+\tilde{x})\|^2.
\end{equation}
If we are in the backtracking step size rule, then $F(x_k)\geq F(\tilde{x})$ and hence $v_k\geq 0$. 
Since also $\delta_k-1=\gamma_k\geq 0$ and $L_{k+1}\geq L_k$, we obtain \beqref{eq:almost_a_k} again from \beqref{eq:abc-epsilon}. \\

{\bf \noindent Step 3:} 
We want to represent the non perturbed term in the left hand side of \beqref{eq:almost_a_k} as 
\begin{equation}\label{eq:a_k-a_k+1}
a_{k}-a_{k+1},
\end{equation}
for some sequence of positive numbers $(a_k)_{k=1}^{\infty}$, and to represent the right hand side of \beqref{eq:almost_a_k} as 
\begin{equation}\label{eq:w_k+1w_k}
\|w_{k+1}\|^2-\|w_k\|^2.
\end{equation}
for a sequence of vectors $(w_k)_{k=1}^{\infty}$. The reason for doing this will become 
clear later (see  \beqref{eq:DecreaseSummable} and the discussion after it). 
For obtaining \beqref{eq:a_k-a_k+1} we impose the condition
\begin{equation}\label{eq:delta_kCond}
\delta_{k+1}(\delta_{k+1}-1)=\delta_k^2,\quad \forall k\geq 1.
\end{equation}
It leads to \beqref{eq:a_k-a_k+1} with
\begin{equation}\label{eq:a_k}
a_k:=2\delta_{k}(\delta_k-1)v_k/L_{k}. 
\end{equation}

{\bf \noindent Step 4:}
For obtaining \beqref{eq:w_k+1w_k} we impose some conditions on the sequence $(y_{k})_{k=2}^{\infty}$ (so far we only assumed the existence of $y_k\in H$ satisfying \beqref{eq:x_kConstStep} or \beqref{eq:x_kBacktrack} but not its form). The condition is that with 
\begin{equation}\label{eq:w_k}
w_{k}:=\delta_k y_{k+1}-(\gamma_k x_k+\tilde{x}),  \quad \forall k\geq 1
\end{equation}
we will have 
\begin{equation}\label{eq:w_k+1}
w_{k+1}=\delta_k x_{k+1}-(\gamma_k x_k+\tilde{x}), \quad \forall k\geq 1.
\end{equation}
Thus, from  \beqref{eq:1+gamma_k=delta_k},\beqref{eq:w_k},\beqref{eq:w_k+1}, 
\begin{multline}\label{eq:y_kExplicit}
y_{k+2}=\frac{w_{k+1}+(\gamma_{k+1} x_{k+1}+\tilde{x})}{\delta_{k+1}}
=\frac{(\delta_{k}x_{k+1}-(\gamma_{k} x_{k}+\tilde{x}))+\gamma_{k+1}x_{k+1}+\tilde{x}}{\delta_{k+1}}\\
=\frac{(\delta_{k+1}+\delta_k-1)x_{k+1}
-(\delta_{k}-1)x_{k}}{\delta_{k+1}},  \quad \forall k\geq 1. 
\end{multline}

{\bf \noindent Step 5:} 
We still need to find $\delta_k$ and $\gamma_k$. After 
 solving the quadratic equation \beqref{eq:delta_kCond} for $\delta_{k+1}$ and taking into account the asumption $\delta_k\geq 0$  for all $k\geq 1$, we obtain 
\begin{equation}\label{eq:delta_k}
\delta_{k+1}=\frac{1+\sqrt{1+4\delta_k^2}}{2},\quad \forall k\geq 1.
\end{equation}
The only restriction on $\delta_1$ is that $\delta_1\geq 1$ so that $\gamma_1\geq 0$ because of \beqref{eq:1+gamma_k=delta_k}. 
There is no restriction on $y_2$. Once we choose $y_2$ and $\delta_1$ we obtain $\gamma_k$ from \beqref{eq:1+gamma_k=delta_k}  
and see that indeed $\gamma_k\geq 0$ and $\delta_k\geq 1$ for all $k$. The equalities and inequalities mentioned earlier indeed hold from the construction of $\delta_k$. 
By denoting 
\begin{equation}\label{eq:t_k=delta_k+1}
t_k:=\delta_{k-1},\quad \forall\,k\geq 2
\end{equation}
we derive the expression mentioned in \beqref{eq:t_k+1}. From \beqref{eq:y_kExplicit} 
we derive the specific form \beqref{eq:y_k+1ConstStep} (and \beqref{eq:y_k+1BackTrack}) of $y_{k+1}$.  \\
 
{\bf \noindent Step 6:} 
Now,  by induction we obtain from  \beqref{eq:epsilon_k+1} and 
\beqref{eq:almost_a_k}-\beqref{eq:w_k} 
that 
\begin{equation}\label{eq:DecreasePerturbe}
a_k+\|w_k\|^2+\epsilon_{k+1}\geq a_{k+1}+\|w_{k+1}\|^2, \quad\forall k\geq 1.
\end{equation}
This implies that the sequence $(a_k+\|w_k\|^2)_{k=1}^{\infty}$ of real numbers 
is decreasing up to a small perturbation. From the inequality $\|e_{k+1}\|\leq s_1$, 
\beqref{eq:x_kBacktrack} (or \beqref{eq:x_kConstStep}), the triangle inequality, \beqref{eq:e_k}, the assumption that $\|\tilde{x}\|\leq \mu$,  \beqref{eq:epsilon_k+1}, and \beqref{eq:t_k=delta_k+1} it follows that for all $k\geq 1$ 
\begin{multline}\label{eq:BoundEpsilon_k+1}
\epsilon_{k+1}=
2 t_{k+1}^2\|e_{k+1}\|\left((\Lambda_{k+1}/L_{k+1}) 
+\|x_{k+1}-((t_{k+1}-1)/t_{k+1})x_k-(1/t_{k+1})\tilde{x}\|\right)\\
\leq 2 t_{k+1}^2\|e_{k+1}\|\left((\Lambda_{k+1}/L_{k+1})+
\|x_{k+1}-((t_{k+1}-1)/t_{k+1})x_k\|+(1/t_{k+1})\mu+2s_1\right)\\
\leq 2 t_{k+1}^2\|e_{k+1}\|\left((\Lambda_{k+1}/L_{k+1})+
\|x_{k+1}\|+\|x_k\|+2s_1+(1/t_{k+1})\mu\right)\\
\leq 2 t_{k+1}^2\|e_{k+1}\|\left((\Lambda_{k+1}/L_{k+1})+\|p_{L_{k+1}}(y_{k+1})\|+s_1
+\|p_{L_{k}}(y_{k})\|+s_1+2s_1+(1/t_{k+1})\mu\right)\\
=\|e_{k+1}\|\sigma_{k+1}\leq s_{k+1}.
\end{multline}
If \beqref{eq:e_k'} holds instead of \beqref{eq:e_k}, then similar considerations 
show that $\epsilon_{k+1}\leq s_{k+1}$ (the third line in \beqref{eq:BoundEpsilon_k+1} is 
replaced by $\|e_{k+1}\|\sigma_{k+1}'\leq s_{k+1}$). Therefore, using \beqref{eq:DecreasePerturbe},
\begin{multline}\label{eq:DecreaseSummable}
a_1+\|w_1\|^2+\sum_{j=2}^{k+1}s_{j}\geq 
a_1+\|w_1\|^2+\sum_{j=1}^{k}\epsilon_{j+1} \geq a_{k+1}+\|w_{k+1}\|^2\geq a_{k+1},\quad \forall k\geq 1.
\end{multline}
The above implies, using \beqref{eq:a_k}, that for all $k\geq 1$ 
\begin{equation}\label{eq:a_1w_1a_k+1}
a_1+\|w_1\|^2+\sum_{j=2}^{k+1}s_{j}\geq a_{k+1}=2t_{k+2}(t_{k+2}-1)(F(x_{k+1})-F(\tilde{x}))/L_{k+1}.
\end{equation}
From \beqref{eq:delta_kCond},\beqref{eq:t_k=delta_k+1}, and \beqref{eq:a_1w_1a_k+1} 
it follows that for all $k\geq 1$ 
\begin{multline}\label{eq:Fx_kt_kL_k+1}
F(x_{k+1})-F(\tilde{x})\leq \frac{L_{k+1}(a_1+\|w_1\|^2+\sum_{j=2}^{k+1}s_{j})}{2t_{k+1}^2}\\
=\frac{L_{k+1}\left((2/L_1)t_2(t_2-1)(F(x_1)-F(\tilde{x}))+
\|t_2y_2-((t_2-1)x_1+\tilde{x})\|^2+\sum_{j=2}^{k+1}s_{j}\right)}
{2t_{k+1}^2}.
\end{multline}

{\bf \noindent Step 7:} From \beqref{eq:delta_k},\beqref{eq:t_k=delta_k+1}, and simple induction  it follows that 
\begin{equation}\label{eq:delta_k geq k}
t_{k+1}=\delta_k\geq 0.5(k+1),\quad \forall k\geq 1.
\end{equation}
This inequality, \beqref{eq:Fx_kt_kL_k+1} and \beqref{eq:tau_rho} yield
\begin{multline}\label{eq:O(1/k^2)}
F(x_{k+1})-F(\tilde{x})\leq
\frac{L_{k+1}(a_1+\|w_1\|^2+\sum_{j=2}^{k+1}s_{j})}{2t_{k+1}^2}\\
\leq\frac{2\tau\left((2/L_1)t_2(t_2-1)(F(x_1)-F(\tilde{x}))+
\|t_2y_2-((t_2-1)x_1+\tilde{x})\|^2+\sum_{j=2}^{k+1}s_{j}\right)}
{(k+1)^2}.
\end{multline}

{\bf \noindent Step 8:} 
It remains to show that under the assumption \beqref{eq:LimSum0} we have  $\lim_{k\to\infty}F(x_k)=\inf_H F$. 
Recall again that either we are in the backtracking step size rule and then $\tilde{x}$ is a minimizer of $F$ or we are in the constant step rule  and then $\tilde{x}$ is arbitrary. In the first case \beqref{eq:O(1/k^2)} and the inequality $\inf_H F=F(\tilde{x})\leq F(x_k)$ for all  $k$ imply the assertion. In the second case we conclude from \beqref{eq:O(1/k^2)} 
that for all $\epsilon>0$ and for all $k$ sufficiently large  
\begin{equation}\label{eq:ApproximateMin}
F(x_k)\leq F(\tilde{x})+\epsilon. 
\end{equation}
Now there are two possibilities: if $\inf_H F=-\infty$, then \beqref{eq:ApproximateMin}, combined with the fact that $\tilde{x}$ was an arbitrary point in $H$, imply that $\lim_{k\to\infty}F(x_k)=-\infty=\inf_H F$, as claimed. Otherwise, we can take $\tilde{x}\in H$ such that $F(\tilde{x})<\inf_H F+\epsilon$ and we conclude that $F(x_k)\leq \inf_H F+2\epsilon$ 
for all $\epsilon>0$ and all $k$ sufficiently large. This and the inequality $\inf_H F\leq F(x_k)$ for all $k$  imply that $\lim_{k\to\infty}F(x_k)=\inf_H F$. 
\end{proof}

\begin{cor}\label{cor:F=O()}
Under the setting of Theorem \bref{thm:FISTA-Perturb}, if there exists a real number $r$ such that $s_k=O(1/k^r)$ for each $k\geq 2$, then 
\begin{equation}\label{eq:RateConvergeFunctionValues}
F(x_{k})-F(\tilde{x})=
\left\{
\begin{array}{lll}
O\left(\displaystyle{\frac{1}{k^{2}}}\right), & \textnormal{if}\,\, r\in(1,\infty),\\
O\left(\displaystyle{\frac{\ln(k)}{k^{2}}}\right), & \textnormal{if}\,\,r=1,\\
O\left(\displaystyle{\frac{1}{k^{1+r}}}\right), & \textnormal{if}\,\,r\in [-1,1).
\end{array}
\right.
\end{equation}
\end{cor}
\begin{proof}
By our assumption there exists $\tilde{c}>0$ such that $s_j\leq \tilde{c}/j^r$ for all $j\in\N$. If $r\in (0,1)$ or $r>1$, then 
\begin{equation*}
\sum_{j=2}^{k+1}s_j\leq\tilde{c}\sum_{j=2}^{k+1}j^{-r}<\tilde{c}\sum_{j=1}^{k-1}\int_j^{j+1}u^{-r}du=\frac{\tilde{c}(k^{1-r}-1)}{1-r}. 
\end{equation*}
If $r=1$, then 
\begin{equation*} 
\sum_{j=2}^{k+1}s_j<\tilde{c}\sum_{j=1}^{k-1}\int_j^{j+1}u^{-1}du=\tilde{c}\ln(k). 
\end{equation*}
If $r\in [-1,0]$, then 
\begin{equation*}
\sum_{j=2}^{k+1}s_j\leq \tilde{c}\sum_{j=2}^{k+1}\int_j^{j+1}u^{-r}du=\frac{\tilde{c}((k+1)^{1-r}-2^{1-r})}{1-r}.
\end{equation*}
By taking into account the above expressions and \beqref{eq:O(1/k^2)} (including the constant terms in the  numerator of \beqref{eq:O(1/k^2)}) we obtain the assertion. 
\end{proof}

\begin{cor}\label{cor:Decay e_k}
Under the assumptions of Theorem \bref{thm:FISTA-Perturb} without assuming \beqref{eq:LimSum0}, we have
\begin{equation}\label{eq:e_kEstimate}
\|e_k\|\leq \frac{s_k}{s_1k^2},\quad \forall\,k\in\N. 
\end{equation}
If, in addition, the following four conditions hold:
\begin{enumerate}[(I)]
\item\label{item:LimF(z_k)=infty} $\lim_{k\to\infty}F(z_k)=\infty$ if and only if $(z_k)_{k=1}^{\infty}$ is an arbitrary sequence in $H$ satisfying $\lim_{k\to\infty}\|z_k\|=\infty$, 
\item\label{item:c} There exists $c\in [0,1]$ such that for all $k\in\N$, $k\geq 2$, if $e_k\neq 0$ and \beqref{eq:e_k} holds, then 
\begin{equation}\label{eq:e_kEquality}
\|e_k\|\geq \frac{c s_k}{\sigma_k},
\end{equation}
and if $e_k\neq 0$ and \beqref{eq:e_k'} holds, then 
\begin{equation}\label{eq:e_kEquality'}
\|e_k\|\geq \frac{c s_k}{\sigma_k'},
\end{equation}
\item\label{item:F>-infty} $\inf\{F(x): x\in H\}>-\infty$, 
\item\label{item:sum s_j bound}
\begin{equation}\label{eq:sup}
\sup\left\{\frac{\sum_{j=2}^{k+1}s_j}{(k+1)^2}: k\in\N\right\}<\infty,
\end{equation}
\end{enumerate}
then, for each $k\geq 2$, either $e_k=0$ or 
\begin{equation}\label{eq:e_k=Theta(s_k/k^2)}
\|e_k\|=\Theta\left(\frac{s_k}{k^2}\right),
\end{equation}
i.e., either $e_k=0$, or, up to a multiplicative constant  factor (independent of $k$) from above and below, $\|e_k\|$ behaves as $s_k/k^2$. In particular, if Conditions \beqref{item:LimF(z_k)=infty}-\beqref{item:sum s_j bound} hold and  if there exists $\omega\in\R$ such that for each $k\geq 2$ either $e_k=0$ or 
\begin{equation}\label{eq:ek-mu^k}
\|e_k\|=\Theta\left(\frac{1}{k^{\omega}}\right),\quad \omega\geq 1, 
\end{equation}
then 
\begin{equation}\label{eq:e_kRateConvergeFunctionValues}
F(x_{k})-F(\tilde{x})=
\left\{
\begin{array}{lll}
O\left(\displaystyle{\frac{1}{k^{2}}}\right), & \omega\in(3,\infty)\\
O\left(\displaystyle{\frac{\ln (k)}{k^{2}}}\right), & \omega=3,\\
O\left(\displaystyle{\frac{1}{k^{\omega-1}}}\right), & \omega\in [1,3).
\end{array}
\right.
\end{equation}
\end{cor}

\begin{proof}
If \beqref{eq:e_k} holds, then from \beqref{eq:e_k} and \beqref{eq:delta_k geq k} we obtain that 
$\|e_k\|\leq 0.5 s_k/(s_1k^2)$. If \beqref{eq:e_k'} holds, then from \beqref{eq:e_k'} and  \beqref{eq:delta_k geq k} we have  $\|e_k\|\leq s_k/(s_1k^2)$. As a result, in any case \beqref{eq:e_kEstimate} holds. 

Assume now that also the other conditions \beqref{item:LimF(z_k)=infty}-\beqref{item:sum s_j bound} hold. From  \beqref{eq:BoundFunctionValues}, Condition \beqref{item:sum s_j bound}, and Condition \beqref{item:F>-infty} it follows that the sequence $(F(x_k))_{k=1}^{\infty}$ is bounded. This and Condition  \beqref{item:LimF(z_k)=infty} imply that  there exists $M>0$ such that  
\begin{equation}\label{eq:M>}
M>\max\{4s_1,2\mu\},\quad\textnormal{and}\,\,\,\|x_k\|<0.5M\,\, \forall\, k\geq 1,
\end{equation}
where $\mu$ is any upper bound on $\|\tilde{x}\|$. This and the triangle inequality show that  the closed ball $B[x_k,2s_1]$ is contained in $B[0,M]$.  Since $F$ is bounded on bounded sets as implied by Conditions \beqref{item:LimF(z_k)=infty} and \beqref{item:F>-infty}, there exists $\Lambda>0$ such that $\Lambda_k<\Lambda$ for all $k\geq 2$, where $\Lambda_k$ is defined in \beqref{eq:Lambda_{k}}. In addition, the above and  \beqref{eq:x_kConstStep} (or  \beqref{eq:x_kBacktrack}) and \beqref{eq:e_k} (or \beqref{eq:e_k'}) imply that 
\begin{equation}\label{eq:pLk<0.5M+s1}
\|p_{L_k}(y_k)\|<0.5M+s_1.
\end{equation}
 Since $t_k\leq t_1k$ for all $k\geq 1$ as implied by a simple induction, it follows from Condition \beqref{item:c}, from \beqref{eq:M>}, from \beqref{eq:tau_rho}, from \beqref{eq:sigma_k}, from \beqref{eq:pLk<0.5M+s1}, and from \beqref{eq:e_kEquality} that for all $k\geq 2$, either $e_k=0$ or  
\begin{equation}\label{eq:e_kOmega}
\|e_k\|\geq \frac{cs_k}{2t_1^2((1/\rho)\Lambda+1.5M+6s_1)k^2}. 
\end{equation}
It follows from \beqref{eq:e_kEstimate} and \beqref{eq:e_kOmega} that for each $k\geq 2$ either $e_k=0$ or $\|e_k\|=\Theta(s_k/k^2)$, as claimed. Similar things can be said if  \beqref{eq:e_k'} and \beqref {eq:e_kEquality'} hold instead of \beqref{eq:e_k} and \beqref{eq:e_kEquality} respectively. 

Finally, assume that Conditions \beqref{item:LimF(z_k)=infty}-\beqref{item:sum s_j bound} hold and that for each $k\geq 2$ either $e_k=0$ or \beqref{eq:ek-mu^k} hold. From what proved above this  implies \beqref{eq:e_k=Theta(s_k/k^2)}. From this and \beqref{eq:ek-mu^k} we have  $s_k=\Theta(1/k^{\omega-2})$. Because $\omega\geq 1$, elementary computations (as in the proof of Corollary \bref{cor:F=O()})  show that \beqref{eq:sup} is not violated. From Corollary \bref{cor:F=O()} we conclude that \beqref{eq:e_kRateConvergeFunctionValues} holds. 
\end{proof}

\begin{remark}\label{rem:t_1y_2}
\begin{enumerate}[(i)]
\item When $t_2=1$ and $s_{k+1}=0$ for all $k\geq 1$, then \beqref{eq:BoundFunctionValues} implies that 
\begin{equation*}
F(x_{k+1})-F(\tilde{x})\leq\frac{2\tau\|y_2-\tilde{x}\|^2}{(k+1)^2}
\end{equation*}
as in \cite[Relation (4.4)]{BeckTeboulle2009jour}, up to the index value (there the index $k$ starts at 0) and up to the fact that $y_1$ in \cite{BeckTeboulle2009jour} is not assumed to be arbitrary as $y_2$ here but is taken to be $x_0$. 
\item Frequently, the expression 
\begin{equation}
\tau_{12}:=2\tau\left((2/L_1)t_2(t_2-1)(F(x_1)-F(\tilde{x}))+
\|t_2y_2-((t_2-1)x_1+\tilde{x})\|^2\right)
\end{equation}
which appears in the right hand side of \beqref{eq:BoundFunctionValues} can be bounded from above  even when $\tilde{x}$ is unknown (when it is a minimizer). For example, consider the $\ell_1$-$\ell_2$ optimization case, i.e., $F(x)=\|Ax-b\|^2+\|x\|_1$, $x\in H:=\R^n$, $A:\R^n\to \R^{n'}$ is linear, $b\in \R^{n'}$. As explained in Remark \bref{rem:sigma_k'}, we have $\|\tilde{x}\|\leq \|b\|^2$. Since $F(\tilde{x})\geq 0$ we conclude from the triangle inequality and the above discussion that 
\begin{equation*}
\tau_{12}\leq2\tau\left((2/L_1)t_2(t_2-1)F(x_1)+(\|t_2y_2-((t_2-1)x_1\|+\|b\|^2)^2\right).
\end{equation*}
\end{enumerate}
\end{remark}

\begin{remark}\label{rem:DecayLiterature}
Interestingly, despite the difference in the various notions of inexactness and the algorithmic schemes  
considered here and elsewhere in the literature, \beqref{eq:e_kRateConvergeFunctionValues},  
as a function of the decay in the error parameters, 
was obtained in \cite[Theorem 2.1]{JiangSunToh2012jour} and \cite[Theorem 4.4]{VillaSalzoBaldassarreVerri2013} (note: in \cite{VillaSalzoBaldassarreVerri2013} the decay in the error parameters is as $\epsilon_k^2$ because of 
\cite[Definition 2.1]{VillaSalzoBaldassarreVerri2013}).  In \cite[Proposition 2]{SchmidtLe-RouxBach2011prep} and the discussion  after it the error parameters were assumed to decay faster in order to achieve 
\beqref{eq:e_kRateConvergeFunctionValues}, e.g., an $O(1/k^4)$ decay for an $O(1/k^2)$ decay 
in the function values.  The algorithmic schemes described in these works include FISTA as a particular case.  
In \cite[p. 62]{DGN2014jour} a slightly better decay rate is given 
in which boundary cases are allowed. For instance, an $O(1/k^3)$ decay implies  
an $O(1/k^2)$ decay in the function values while we require a $\Theta(1/k^{3+\beta})$ decay 
for arbitrary $\beta>0$. However, as mentioned at the end of Subsection \bref{subsec:Background}, the setting in \cite{DGN2014jour} is somewhat different from our one, especially when a separable function is considered. Interestingly, even in  \cite{MonteiroSvaiter2013jour}, which, as explained in Subsection \bref{subsec:Background}, also  considers a different setting from our one, one can find traces of the decay rate $O(1/k^3)$ of the errors: see \cite[Proposition 5.2(c)]{MonteiroSvaiter2013jour}. 

The above discussion leads us to conjecture that there are some non-obvious relations between the various notions of inexactness. In fact, \cite[Proposition 2.5]{SalzoVilla2012jour} and the discussion after  \cite[Definition  2.1]{SalzoVilla2012jour} shows that our notion of inexactness may be weaker than the one discussed in  \cite{VillaSalzoBaldassarreVerri2013}. On the other hand, because in Corollary \bref{cor:Decay e_k} we impose Conditions  \beqref{item:LimF(z_k)=infty}-\beqref{item:sum s_j bound}, we assume something which is not assumed  in \cite{VillaSalzoBaldassarreVerri2013} and in other works mentioned above (Corollary \bref{cor:Decay e_k} is especially good for the case of superiorization because in this case the user actively controls the errors). We also suspect that there are examples for functions $F$ such that  $\|e_k\|=\Theta(1/k^{\omega})$ for fixed $\omega\in (0,1)$ but $\lim_{k\to\infty}F(x_k)$ does not exist or it exists but is not equal to $F(\tilde{x})$ (assuming $\tilde{x}$ is a minimizer of $F$).
\end{remark}

\section{Superiorization}\label{sec:Superiorization}
\subsection{Background} In Section \bref{sec:Intro} we mentioned briefly the superiorization methodology as one of the reasons for considering inexact versions of FISTA. Motivated by this reason, we re-examine in this section the superiorization  methodology in a thorough way and  show that its scope can be significantly extended. 

First, let us recall again the principles behind the superiorization methodology. Suppose that  our goal is to solve some constrained optimization problem. The full problem might be too  demanding from the computational point of view, but solving only the constrained part (the feasibility problem) can be achieved by an algorithm $\mathcal{A}$ which is rather simple and computationally cheap.  Suppose further  that $\mathcal{A}$ is known to be perturbation resilient, that is, a perturbed version $\mathcal{A}'$ of $\mathcal{A}$ due to error terms also produces solutions to the constrained  part. The superiorization methodology claims that  often we can do something useful with the perturbed version. The ``something  useful'' can be a solution $x'$ (or an approximation solution) to  the feasibility problem which is superior, with  respect to some given cost function $\phi$, to a solution $x$ which would be obtained by considering  the  original algorithm $\mathcal{A}$. In other words, $\phi(x')\leq \phi(x)$, and frequently $\phi(x')$ is much smaller than $\phi(x)$ or at least the computation time needed to find $x'$ will be smaller than the one needed to find $x$. A possible way to approximate $x'$ is by performing in each iteration a feasibility seeking-step and immediately after it a superiorization step aiming at reducing $\phi$ at the current iteration by  playing carefully with the error parameters.  

This heuristic methodology was officially introduced in 2009 in \cite{DavidiHermanCensor2009jour}, but historically, the first works in this research branch are  the 2007 paper  \cite{ButnariuDavidiHermanKazantsev} and the 2008 paper \cite{HermanDavidi2008jour} which did not use the  explicit term ``superiorization''. Since then, the methodology has been investigated in various works, e.g., in  \cite{BauschkeKoch2015jour,CensorDavidiHerman2010jour,CDHST2014,CensorReem2015jour,CensorZaslavski2015jour,Davidi2010PhD,DCSGX2015jour,HGDC2012jour,JinCensorJiang2013proc,JinCensorJiang2016jour,Langthaler2014report,PSCR2010jour,Prommegger2014report}. See also \cite{Censor2015surv,Herman2014surv} for two recent surveys and \cite{CensorSuperiorizationPage}  for a continuously updated online list of works related to the superiorization methodology.  Although the point $x'$ is not a solution to the original constrained  optimization problem, promising experimental results discussed in many of the above mentioned works show the  potential of superiorization in real-world scenarios (for instance, for the analysis of images coming from medical sciences and machine  engineering). 

However, from the theoretical point of view the methodology is still in its initial stages. In  particular, the few mathematical results that exist do not give a full theoretical justification of its success. As a matter of fact, even the potential scope of the methodology has not been fully investigated.   So, on the one hand, some of these works (e.g.,  \cite{Censor2015surv,Davidi2010PhD,Herman2014surv} and \cite{CensorDavidiHerman2010jour}) show that the pioneers of this  methodology have definitely been aware of the generality of the approach, but, on the other hand, a more  careful reading of these works (e.g., Definition 4, Algorithm 5, and Definition 9 in \cite{Censor2015surv}) show that the actual setting which has been considered is not completely general.

To be more concrete, the setting  is a real Hilbert space $H$ (usually finite dimensional); the perturbed iterations should have the form $x_{k+1}=T_k(x_k+\beta_k v_k)$ for some operator  $T_k:H\to H$, where $(v_k)_{k=1}^{\infty}$ is a bounded sequence in $H$ and $(\beta_k)_{k=1}^{\infty}$ is a sequence of nonnegative real numbers satisfying $\sum_{k=1}^{\infty}\beta_k<\infty$; if a convergence notion is discussed, then this notion is  standard: mainly strong convergence (rarely, as in \cite{CensorReem2015jour}, also convergence in the weak topology); 
in several places, e.g., \cite{CensorDavidiHerman2010jour}, there are limitations on the considered functions (e.g., $\phi$ must be convex); the algorithmic operator used at iteration $k+1$ depends only on iteration $k$ (and possibly on some parameters depending on $k$) but not on  previous iterations such as both iterations $k$ and $k-1$, as, e.g., in the perturbed version of FISTA (Section \bref{sec:FISTA-Perturbed}). 

Moreover, in all the works related to superiorization that we have seen, the  perturbation resilience property of the algorithm $\mathcal{A}$ mentioned above has been understood in the feasibility sense and not in other contexts (e.g., in a context of finding a superior solution to an unconstrained minimization problem using a perturbation  resilient algorithm). In other words, the  perturbation resilience property is understood in the sense that both the sequence produced by $\mathcal{A}$ and its  perturbed version  produced by $\mathcal{A}'$ should  converge to a feasible solution. A frequently used version of this criterion is to use a proximity function which measures the distance to the feasible set, and a solution is a point in the space in which this proximity function attains a value not greater than some given error parameter \cite{Censor2015surv,Davidi2010PhD,DCSGX2015jour,GardunoHerman2014jour}. The above is consistent with the fact that often the  superiorization methodology is described as lying between optimization and the  (convex) feasibility problem: see, e.g.,   \cite{Censor2015surv,CensorDavidiHerman2010jour,CensorZaslavski2015jour,PSCR2010jour} 
and \cite[p. 90]{Davidi2010PhD}.
 
\subsection{Our contribution} What is suggested here is to extend the superiorization principle by allowing any type of perturbations, any notion of inexactness, any notion of convergence, and any type of optimization-related problem. More precisely,  given any optimization-related problem in some given space, suppose that we have in our hands a notion of an algorithm $\mathcal{A}$ which produces  a sequence of elements  in the space (they can be thought of as being intermediate solutions to the problem) and a notion of a solution  of the problem (e.g., the  limit of the sequence or some intermediate solution satisfying a certain termination criterion). Moreover, suppose that we have in our hands a notion of inexactness (or a notion of perturbation) of the algorithm, so that instead of considering the sequence produced by $\mathcal{A}$ we  consider a sequence produced by a perturbed algorithm $\mathcal{A}'$. If there is a mathematical result saying that any perturbed  sequence (according to our notion of inexactness) also induces a solution to the original problem, then we can consider the set of all perturbed sequences, with the hope that we will be able to find in this set, by one way or another, a sequence which will lead us to a superior solution. 
 
Roughly speaking, a ``superior solution'' means a solution to the original problem which is  better, according to some criterion (preferably a criterion which is  quantitative and simple to apply),  than ``standard solutions'', namely, solutions which are found using the algorithm $\mathcal{A}$. This additional criterion can be thought of as being ``a notion of superiority''. For example, the notion of superiority can be based on a given cost function $\phi$. In this case, if $(x_k)_k$ is the sequence produced by the original algorithm $\mathcal{A}$ with an induced solution $x$, and if $(x_k')_k$ is the perturbed sequence having  $x'$ as the induced solution, then $x'$ is considered as being a superior solution to $x$ if $\phi(x')\leq \phi(x)$. Alternatively, we can say that $x'$ is superior to $x$ whenever  $\phi(x'_k)\leq \phi(x_k)$ for all $k$ large enough. In both cases strict inequalities are preferred. When the  original problem is to minimize a function $F$ under some constraints, then a possible choice for $\phi$ is to take $\phi:=F$. A third superiority criterion is to consider   several cost  functions $\phi_i,\,i\in I$ for some nonempty set of indices $I$, i.e., $\phi_i(x')\leq \phi_i(x)$ for all $i\in I$, or at least  that $\phi_i(x'_k)\leq \phi_i(x_k)$ for all $i\in I$ and all $k$ large enough. A simple illustration for this third criterion is to take $I=\{1,2\}$, $X=\R^n$, $\phi_1:X\to[0,\infty)$ as the total variation and $\phi_2:X\to [0,\infty)$ as the penalty function $\psi$ suggested in \cite{LevitanHerman1987jour} (see also \cite[p. 166]{GardunoHerman2014jour}).

In practice the perturbed sequence $(x_k')_k$  will be determined by some error terms (which can be vectors, positive parameters, etc.). No matter how we play with these error terms, as long as they satisfy the conditions of the perturbation resilient result that we have in our hands, we obtain a sequence which is guaranteed to  converge in some sense to a solution of the problem. However, by a clever modification of the error terms in each iteration we may steer the sequence to a superior  solution. 

The examples below show  the wide spectrum of this general principle (virtually, any optimization-related problem can be considered), thus significantly extending the scope of the original  superiorization methodology. In order to simplify the notation below, we refer to the 
error terms as $e_k$ when they are vectors and $\epsilon_k$ when they are positive numbers  (although in the original works a different notation was sometimes used).

\begin{expl}
Optimization problem: (accelerated) minimization of a convex function in finite and infinite  dimensional Hilbert spaces.  Notion of convergence: non-asymptotic (function values). A few notions of inexactness:  see the details regarding Devolder et al \cite{DGN2014jour}, Jiang et al. \cite{JiangSunToh2012jour}, Monteiro-Svaiter \cite{MonteiroSvaiter2013jour}, Schmidt et al \cite{SchmidtLe-RouxBach2011prep},  and Villa et al.  \cite{VillaSalzoBaldassarreVerri2013}  in Section \bref{sec:Intro} above; see also \beqref{eq:x_kConstStep} and Theorem \bref{thm:FISTA-Perturb} above.
\end{expl} 

\begin{expl}
Optimization problem: finding zeros of (nonlinear, maximal monotone) operators. Notion of convergence: weak or strong topology. A few notions of inexactness and settings:  
\begin{itemize}
\item Rockafellar \cite{Rockafellar1976jour}:  $\|x_{k+1}-P_k(x_k)\|\leq  \epsilon_k$ or  $\|x_{k+1}-P_k(x_k)\|\leq  \epsilon_k\|x_{k+1}-x_k\|$, where  $\sum_{k=1}^{\infty}\epsilon_k<\infty$ and $P_k=(I+c_kT)^{-1}$ is a proximal operator induced by the operator $T$ whose zeros are sought and $c_k>0$. Setting: a real Hilbert space. 
\item Eckstein \cite{Eckstein1998jour}: $\nabla h(x_k)+e_k\in \nabla h(x_{k+1})+c_kT(x_{k+1})$ for a given Bregman function $h$, where both $\sum_{k=1}^{\infty}\|e_k\|<\infty$ and  $\sum_{k=1}^{\infty}\langle e_k,x_k\rangle$ should exist and be finite. Setting: the Euclidean $\R^n$.
\item Solodov-Svaiter \cite{SolodovSvaiter2001jour}: here the goal is to find a zero of the operator $T$ in a real Hilbert space under a linear constraint. The perturbation appears in several forms: first, in an  $\epsilon_k$-enlargement of $T$; second, in a certain inequality involving $\epsilon_k$, $x_k$,  and other components of the algorithm (including a relative error tolerance $\sigma_k$); third, in an ``halfspace-type projection'' $a_k$ involving $\epsilon_k$. 

\item Reich-Sabach \cite{ReichSabach2010b-jour}: here the goal is to find a common zero of  finitely many operators $A_i$, $i\in \{1,\ldots,N\}$ in a real reflexive Banach space. There are two types of perturbations. 
The first type appears in \cite[(4.1)]{ReichSabach2010b-jour} in four places. The first place is 
in the  equation $e^i_k=\xi^i_k+\frac{1}{\lambda^i_k}\left(\nabla f(y^i_k)-\nabla f(x_k)\right)$ where $\xi^i_k\in A_i(y^i_k)$, the second is in the term  
$w^i_k=\nabla f^*\left(\lambda^i_k e^i_k+\nabla f(x_k)\right)$, the third is in the set 
$C^i_k=\{z\in X: D_f(z,y^i_k)\leq D_f(z,w^i_k)\}$ via $w^i_k$, and the fourth is in the set $C_k:=\cap_{i=1}^N C^i_k$. Here $f$ is a Bregman function, $D_f$ is the induced Bregman divergence (Bregman distance), $\lambda^i_k$ is a positive parameter, $f^*$ is the convex conjugate (Fenchel conjugate) of $f$,  and $y^i_k$ is an additional term satisfying certain relations. 

The  second type appears in \cite[(4.4)]{ReichSabach2010b-jour} in three places. The first place is in the term  $y^i_k=\textnormal{Res}^f_{\lambda^i_kT_i}(x_k+e^i_k)$ where $f$ is a Bregman function, $\lambda^i_k$ is a certain positive parameter, $x_k$ is determined in 
other steps of the algorithm, and $\textnormal{Res}^f_{\lambda^i_kT_i}$ is the resolvent of the operator $\lambda^i_k T_i$ relative to $f$. The second place is in the definition of a certain subset  $C^i_k$ defined in an intermediate step of the algorithm and the perturbation appears as  $x_k+e^i_k$ inside the definition of $x^i_k$. The third place is in the set $C_k:=\cap_{i=1}^N C^i_k$. The error terms $e^i_k$ can be arbitrary (this issue has been clarified recently and will be discussed elsewhere).
\end{itemize}
\end{expl} 

\begin{expl}
Optimization problem: finding fixed points of nonlinear operators in real reflexive Banach spaces. Notion of convergence: weak or strong topology. Some examples: 
\begin{itemize}
\item Reich-Sabach \cite{ReichSabach2010jour}: here the goal is to find a common fixed point of  finitely many operators $T_i$, $i\in \{1,\ldots,N\}$.  The perturbation comes in two forms:  first, as  $y^i_k=T_i(x_k+e^i_k)$ where  $x_k$ is determined in other intermediate steps of the algorithm. Second, the perturbation also appears (as $x_k+e^i_k$) in the definition of a certain subset $C^i_k$ defined in an intermediate step of the algorithm. The error terms $e^i_k$ can be arbitrary (this issue has been clarified recently and will be discussed elsewhere).
\item Butnariu-Reich-Zaslavski \cite{ButnariuReichZaslavski2006conf}: here several notions of  inexactness are used. These conditions are equivalent to saying that four sequences  $(\epsilon_{i,k})_{k=1}^{\infty}$, $i\in \{1,2,3,4\}$ of nonnegative numbers are given and we assume that  their   sum is finite; now, for each $k\in\N$ the iteration $x_{k+1}$ is an arbitrary vector which satisfies the following inequalities: $D_f(T(x_k),x_{k+1})\leq \epsilon_{1,k}$, $\|f'(T(x_k))-f'(x_{k+1})\|\leq \epsilon_{2,k}$,
$\|f'(T(x_k))-f'(x_{k+1})\|\|T(x_k)\|\leq \epsilon_{3,k}$, and $\langle  f'(x_{k+1})-f'(T(x_{k})),x_{k+1}-T(x_k)\rangle\leq \epsilon_{4,k}$. Here $T$ is the operator  whose fixed point are sought and $D_f$ is a Bregman divergence (distance) with respect to a given  Bregman function $f$. 
\end{itemize}
\end{expl}

\begin{expl}
Optimization problem: minimization of a real lower semicontinuous proper convex function $f$. We mention here two examples: 
\begin{itemize}
\item Cominetti \cite{Cominetti1997jour}: The notion of convergence is weak or strong. Notion of inexactness: $x_k-(1/\lambda_k)x_{k-1}\in \partial_{\epsilon_k}f(x_k,r_k)$, where $\partial_{\epsilon_k}$ is an $\epsilon_k$-subdifferential (of $f(\cdot,r_k)$),  $f(\cdot,\cdot)$ is (by abuse of notation) a proper convex lower semicontinuous  approximation of $f$ depending on $x_k$, $\lambda_k>0$, and $r_k>0$ and has the property that its minimal value is finite and tends  to the minimal value of $f$ (whose set of minimizers is assumed to be nonempty) as $r>0$ tends to 0. There are a few conditions on some parameters, e.g., in \cite[Theorem 3.1]{Cominetti1997jour} one requires that  $\lim_{k\to\infty}r_k=0$, $\sum_{k=1}^{\infty}\beta(r_k)\lambda_k=\infty$,   $\sum_{k=1}^{\infty}\epsilon_k\lambda_k<\infty$ or $\lim_{k\to\infty}\epsilon_k/\beta(r_k)=0$, and there are additional conditions; here $\beta(r_k)>0$ is a strong convexity parameter of  $f(\cdot,r_k)$. Setting: a real  Hilbert space.
\item  Zaslavski \cite{Zaslavski2010b-jour}: two notions of convergence are used: in the first  \cite[Theorem  1.2]{Zaslavski2010b-jour} the notion is that the distance of $x_k$ from the solution set is smaller than a given error parameter $\epsilon>0$. The second notion of convergence is convergence in the function values. The notion of inexactness in both cases has the form  $x_k+e_k=\argmin_{x\in\R^n}(f(x)+(1/\lambda_{k-1})B(x,x_{k-1}))$ for some Bregman divergence $B$ and a relaxation parameter $\lambda_{k-1}>0$, $k\in\N$. In addition, it is  assumed that there exists $\delta>0$ depending on $\epsilon$ such that $\|e_k\|\leq \delta$ for each $k\in\N$. Setting: the Euclidean $\R^n$.
\end{itemize}
\end{expl}   

\begin{expl}
Optimization problem: a generalized mixed variational inequality problem in a real Hilbert space   (Xia-Huang \cite{XiaHuang2011jour}). Notion of convergence: weak. Notion of inexactness: based on error terms  whose magnitude should be small enough so that it satisfies a certain 
implicit inequality \cite[Relation (3.3)]{XiaHuang2011jour} which is also determined by some parameters given by the user including a relative error parameter $\sigma$. 
\end{expl}

\begin{expl}
Optimization problem: finding attracting points of an infinite product of countably many  nonexpansive operators $T_i$, $i\in\N$ (Pustylnik-Reich-Zaslavski \cite{PustylnikReichZaslavski2009jour}) in a complete metric space. Notion of convergence: the distance between the iterations and the attracting set $F$ tends to 0. Notion of inexactness: for each $\epsilon>0$ there exists $\delta>0$ and a natural number $n_0$ such that for each ``good'' control $r:\{0,1,2,\ldots\}\to\{0,1,2,\ldots\}$ and each sequence $(x_{k})_{k=0}^{\infty}$  satisfying  $d(x_{k+1},T_{r(k)}x_k)\leq \delta$ for each $k\in\{0,1,2,\ldots\}$, the inequality  $d(x_k,F)<\epsilon$ holds for every $k\geq n_0$. 
\end{expl}

\begin{expl}
Optimization problem: solving the (convex) feasibility problem. Many examples are given in papers dealing with superiorization. Here we mention examples which seem to be less familiar in the superiorization literature. The notion of inexactness in them is weak or strong. 
\begin{itemize}
\item De Pierro-Iusem \cite{DePierroIusem1988}: the perturbation appears as  $x_{k+1}=x_k-\displaystyle{\frac{\alpha_k(g_{i(k)}(x_k)+\epsilon_k)}{\|t_k\|^2}t_k}$ when $g_{i(k)}(x_k)>0$; here $t_k\neq 0$ is a  subgradient of the convex function $g_{i(k)}$ at the point $x_k$ and $\alpha_k$ is a relaxation parameter. It is assumed \cite[Section 3.1]{DePierroIusem1988} that $(\epsilon_k)_{k=1}^{\infty}$ is a monotonically decreasing sequence  of positive parameters which converges to zero and satisfies the condition  $\sum_{k=1}^{\infty}\epsilon_k=\infty$. Setting: the Euclidean $\R^n$. 
\item Censor-Reem \cite{CensorReem2015jour}: the perturbation has the form $P_{\Omega}\left(x_{k}-\displaystyle\lambda_{k}{\frac{g_{i(k)}(x_{k}
)}{\parallel t_{k}\parallel^{2}}}t_{k}+e_{k}\right)$ whenever $g_{i(k)}(x_{k})>0$; 
here $t_k\neq 0$ is a zero-subgradient of the zero-convex function $g_{i(k)}$ at the point $x_k$ and $\lambda_k>0$ is a relaxation parameter, and $P_{\Omega}$ is the best approximation projection on the nonempty closed and convex subset $\Omega$ on which the functions $g_j$, $j\in\N$ are defined. There are additional assumptions, among them \cite[Condition 1]{CensorReem2015jour} saying that for each $k\in\N$ the norm of the error term $e_k$ is bounded above by $\min\{\mu,\epsilon_1\epsilon_2 h_k^2/(2(5\mu+4h_k))\}$, where  $\mu$, $\epsilon_1$, and $\epsilon_2$ are  certain given positive parameters and $h_k$ is a certain nonnegative parameter depending on other parameters (e.g., on $g_{i(k)}(x_k)$).  For a slightly different type of perturbation, see \cite[Subsection 8.1]{CensorReem2015jour}. Setting: a real Hilbert space. 
\end{itemize}
\end{expl} 

\begin{expl}
Optimization problem: any problem which makes use of relaxation parameters (as in many of the 
above examples). These parameters can also be thought of as ``resilience error parameters'' 
since it is guaranteed that the various algorithms converge whenever the  
parameters satisfy a mild condition (e.g., being in the interval $(\epsilon,2-\epsilon)$ for some arbitrary small $\epsilon\in (0,1)$). It is well-known that the relaxation parameters can significantly influence the speed of convergence of the algorithm (for a simple illustration of  this phenomenon, see \cite[Section 7]{CensorReem2015jour}).
\end{expl}

Many additional examples can be found in the following rather partial list of references and in  some of the references therein:  \cite{AlberBurachikIusem1997,ArgyrosMagrenan2015jour,BurachikSvaiter1999jour,BurkeQian1999jour,Combettes2004jour,CombettesWajs2005jour,DemboEisenstatSteihaug1982jour,FriedlanderSchmidt2012jour,Guler1992jour,HeYuan2012jour,HumesSilva2005jour,IusemOtero2001jour,IusemPennanenSvaiter2003,KangKangJung2015jour,KaplanTichatschke2004jour,LWWLW2016jour,MonteiroSvaiter2012jour,ParenteLotitoSolodov2008jour,ReichSabach2009jour,ReichSabach2012col,ReichZaslacski2014jour,SalzoVilla2012jour,SantosSilva2014jour,SolodovSvaiter1999-2jour,SolodovSvaiter1999-1jour,SolodovSvaiter2000jour,Tran-DinhNecoaraDiehl2016jour,Zaslavski2010jour,Zaslavski2011jour,
Zaslavski2013b,Zaslavski2013a,Zaslavski2014jour}. Most of the above mentioned references 
do not mention the word ``superiorization'' explicitly. In fact, 
many of the involved authors had not even been aware of this optimization branch at the time of  preparation of their papers (e.g., because many papers were published years before the superiorization methodology was introduced). However, as said above, one can find in these papers results ensuring the perturbation resilience of certain algorithms. One can also think about other settings in which the superiorization methodology can be used, e.g., when the notion of  convergence  is based on Banach limits, asymptotic centers, convergence in the sense of Mosco, etc., and when the optimization problems are combinatorial or mixed combinatorial (integer  programming) and continuous.

\subsection{Concluding remarks} 
We want to conclude this section with the following words. The previous paragraphs not only extend the horizon of the superiorization methodology, but also pose various challenges. First, to develop a formalism which will handle the above mentioned  examples (or at least an important class of them) in a rigorous way. Second, to provide various real world examples showing the usefulness of the general superiorization methodology. Third, to formulate theoretical and practical sufficient (and/or necessary) conditions which will ensure the convergence (in the considered notion of convergence) of the perturbed sequence to a superior solution. Fourth, to obtain results regarding rates of convergence (e.g., that given some approximation parameter $\epsilon>0$, there exists $k_{\epsilon}\in\N$ such that for all $k_{\epsilon}\leq k\in \N$  iteration number $k$ of the perturbed algorithm is an $\epsilon$-solution of the original problem). Fifth, to obtain theoretical and practical results for multiple  cost functions (this creates an   interesting and new connection between superiorization and feasibility, where this time a feasibility is not the target of the perturbed algorithm, but rather an assumption about the existence of a joint superior solution for several cost functions). Sixth, to present systematic methods for finding good perturbations, e.g, ones which  will ensure that with high probability the perturbed iteration is superior  to the unperturbed one. It is our hope that at least some of these challenges will be addressed and that the discussion of this section will be found to be helpful in optimization theory and beyond.

\section{Appendix}\label{sec:Appendix}
In this appendix we present the proofs of a few auxiliary claims mentioned in the main body of the text. Lemma \bref{lem:LipschitzUpperBound} below was mentioned in Remark \bref{rem:L_kFiniteFx_k}. 

\begin{lem}\label{lem:LipschitzUpperBound}
Given a real Hilbert space $H$ with an inner product $\langle \cdot,\cdot\rangle$ and an induced norm $\|\cdot\|$, suppose that $f:H\to \R$ is continuously differentiable with a Lipschitz constant $L(f')$ of $f'$. Then for all $L\geq L(f')$ and all $x,y\in H$
\begin{equation}\label{eq:fLApp}
f(x)\leq f(y)+\langle f'(y), x-y\rangle+0.5L\|x-y\|^2. 
\end{equation}
\end{lem}
\begin{proof}
Fix $x,y\in H$ and let $\phi:[0,1]\to\R$ be defined by $\phi(t)=f(y+t(x-y))$. From the chain rule  $\phi'$  is continuous and $\phi'(t)=\langle f'(y+t(x-y)),x-y\rangle$ for each $t\in  [0,1]$. 
 As a result, the fundamental theorem of calculus, the 
  assumption  that $f'$ is Lipschitz continuous,  the triangle inequality for integrals, and the Cauchy-Schwarz inequality imply that 
\begin{multline*}
f(x)=\phi(1)=\phi(0)+\int_0^1\phi'(t)dt=f(y)+\int_0^1\langle f'(y+t(x-y)),x-y\rangle dt\\
=f(y)+\int_0^1\langle f'(y),x-y\rangle dt+\int_0^1\langle f'(y+t(x-y))-f'(y),x-y\rangle dt\\
\leq f(y)+\langle f'(y),x-y\rangle+\int_0^1|\langle f'(y+t(x-y))-f'(y),x-y\rangle| dt\\
\leq f(y)+\langle f'(y),x-y\rangle+\int_0^1\|f'(y+t(x-y))-f'(y)\|\|x-y\| dt\\
\leq f(y)+\langle f'(y),x-y\rangle+\int_0^1 L \|y+t(x-y)-y\|\|x-y\| dt\\
=f(y)+\langle f'(y),x-y\rangle+L\|x-y\|^2\int_0^1 tdt \\
=f(y)+\langle f'(y),x-y\rangle+0.5L\|x-y\|^2.
\end{multline*}
\end{proof}

Lemma \bref{lem:Optimality} below is needed for proving Lemma \bref{lem:FB_k}. 
\begin{lem}\label{lem:Optimality}
Let $H$ be a real Hilbert space with an inner product $\langle \cdot,\cdot\rangle$ and an induced norm $\|\cdot\|$. For all $y\in H$ and $L>0$, let $u:H\to(-\infty,\infty]$ be defined by  $u(x):=Q_{L}(x,y)$, where $Q_{L}$ is defined in  \beqref{eq:Q_L} with $L$ instead of $L_k$. Then $u$ has a unique minimizer  $p_{L}(y)$ and there exists $\gamma\in \partial g(p_{L}(y))$ such that 
\begin{equation}\label{eq:z_gamma}
f'(y)+\gamma=L(y-p_{L}(y)).
\end{equation}
\end{lem}
\begin{proof}
Since $g$ is proper, convex, and lower semicontinuous, it  follows from the definition of $u$ and $Q_L$ that $u$ is the sum of the smooth convex and quadratic function $q(x):= f(y)+\langle f'(y),  x-y\rangle+0.5L\|x-y\|^2$ and the proper convex lower  semincontinuous function $g$. 
Hence by \cite[Corollary 11.15]{BauschkeCombettes2011book} there exists a unique global minimizer  $p_{L}(y)$ of $u$. By Fermat's rule  \cite[Theorem 16.2, p. 233]{BauschkeCombettes2011book} 
a point $z$ is a (global) minimizer of some proper function $G$ if and only if $0\in \partial G(z)$. Let $G:=u$ and $z:=p_L(y)$. Since $q$ is differentiable, from \cite[Proposition 17.26, p. 251]{BauschkeCombettes2011book} one has $\partial q(x)=\{q'(x)\}$ for each $x\in H$. Since $0\in \partial G(z)$, the sum rule \cite[Theorem 5.38, p. 77]{VanTiel1984} and its proof  imply that $\partial g(z)\neq\emptyset $ and  $\partial G(z)=\partial q(z)+\partial g(z)$. The assertion follows from the above lines because $q'(z)=f'(y)+L(z-y)$. 
\end{proof}

\begin{proof}[{\bf Proof of Lemma \bref{lem:FB_k}}]
Since we can use Lemma \bref{lem:Optimality} in our infinite dimensional setting, the proof is very similar to the proof of  \cite[Lemma 2.3]{BeckTeboulle2009jour}. Indeed, fix $x\in H$. From  the inequality 
$F(p_L(y))\leq Q_L(p_L(y),y)$ we have 
\begin{equation}\label{eq:F>=Q}
F(x)-F(p_L(y))\geq F(x)-Q_L(p_L(y),y).
\end{equation}
Since $f'$ exists, $\partial f(x)=\{f'(x)\}$ for each $x\in H$ as follows from \cite[Proposition 17.26, p. 251]{BauschkeCombettes2011book}. From Lemma \bref{lem:Optimality} we know that the exists $\gamma\in \partial g(p_L(y))$ such that \beqref{eq:z_gamma} holds. The above and the subgradient inequality imply the following inequalities:  
\begin{equation*}
f(x)\geq f(y)+\langle f'(y),x-y\rangle,
\end{equation*}
\begin{equation*}
g(x)\geq g(p_L(y))+\langle \gamma,x-p_L(y)\rangle.
\end{equation*}
After summing these inequalities and recalling that $F=f+g$ we arrive at
\begin{equation}\label{eq:F>=}
F(x)\geq f(y)+\langle f'(y),x-y\rangle+g(p_L(y))+\langle \gamma,x-p_L(y)\rangle.
\end{equation}
From  \beqref{eq:Q_L} one has
\begin{equation}\label{eq:Qp}
Q_L(p_L(y)),y)=f(y)+\langle f'(y),p_L(y)-y\rangle+0.5L\|p_L(y)-y\|^2+g(p_L(y)).
\end{equation}
As a result of \beqref{eq:z_gamma} and \beqref{eq:F>=Q}-\beqref{eq:Qp} we have
\begin{multline*}
F(x)-F(p_L(y))\geq \langle x-p_L(y),f'(y)+\gamma\rangle-0.5L\|p_L(y)-y\|^2\\
=\langle x-p_L(y),L(y-p_{L}(y))\rangle-0.5L\|p_L(y)-y\|^2\\
=\langle y-p_L(y),L(y-p_{L}(y))\rangle+\langle x-y,L(y-p_{L}(y))\rangle-0.5L\|p_L(y)-y\|^2\\
=L\|y-p_L(y)\|^2+L\langle x-y,y-p_{L}(y)\rangle-0.5L\|p_L(y)-y\|^2\\
=0.5L\|p_L(y)-y\|^2+L\langle y-x,p_L(y)-y\rangle
\end{multline*}
as claimed. 
\end{proof}

\section*{Acknowledgments}
We would like to thank FAPESP 2013/19504-9 for supporting this work. Alvaro De Pierro wants to thank CNPq  grant 306030/2014-4. We would like to express our thanks to Jose Yunier Bello Cruz,  Yair Censor, Gabor Herman, Simeon Reich, and Shoham Sabach for helpful discussions. We also thank the referees for considering the paper and for their feedback. 
\bibliographystyle{amsplain}
\bibliography{biblio}

\end{document}